\documentclass[11pt]{article}

\usepackage[cp1252]{inputenc} %en windows mettre [cp1252]{inputenc}
\usepackage[T1]{fontenc}
\usepackage[a4paper]{geometry}
\usepackage[frenchb,english]{babel}

\usepackage{amsmath, amssymb,mathrsfs,empheq}
\usepackage{amsthm}
\usepackage{graphicx}
\usepackage{fancybox}
\usepackage{morefloats}
\usepackage{color}

\allowdisplaybreaks[4]

\theoremstyle{plain}

\newtheorem{prop}{Proposition}
\newtheorem{thm}{Theorem}
\theoremstyle{definition}

\theoremstyle{remark}
\newtheorem*{rmq}{Remark}

\newcommand{\R}{\mathbb{R}}
\newcommand{\dd}{\mathrm{d}}
\newcommand{\ep}{\varepsilon}
\newcommand{\alp}{\alpha}

\newcommand{\e}{\mathrm{e}}
\newcommand{\dt}{\Delta t}
\newcommand{\dx}{\Delta x}
\newcommand{\dv}{\Delta v}
\newcommand{\lla}{\left\langle}
\newcommand{\rra}{\right\rangle}

\renewcommand{\phi}{\varphi}

\newcommand{\ccl}{[\![}
\newcommand{\ccr}{]\!]}
\newcommand{\tin}{\mathrm{in}}

\title{An asymptotic preserving scheme for front propagation in a kinetic reaction-transport equation}
\author{Hélène Hivert
\footnote{\textbf{H\'elène \bsc{Hivert}}: UMPA , 
ENS de Lyon, site Monod. $46$ allée d'Italie. $69007$ Lyon.
helene.hivert@ens-lyon.fr } 
  }
\date{May $10$, $2017$}

\begin{document}

\selectlanguage{english}
\maketitle

 \begin{abstract}
In this work, we propose an asymptotic preserving scheme for a non-linear kinetic reaction-transport equation, 
%which is 
%in the large deviations regime.
in the regime of sharp interface.
%able to deal with a Hamilton-Jacobi limit. The model has been introduced to study the formation of pulse waves in a colony of bacteria \emph{E. coli} at the mesoscopic scale. 
%This
%A phenomenon of front propagation can be observed when considering the kinetic equation under an hyperbolic scaling, and using a Hopf-Cole transform for the distribution function. 
With a non-linear reaction term of KPP-type, a phenomenon of front propagation has been proved in \cite{BouinCalvezNadin}. This behavior can be highlighted by considering a suitable hyperbolic limit of the kinetic equation, using a Hopf-Cole transform.
It has been proved in \cite{Bouin, BouinCalvez, Caillerie} that the logarithm of the distribution function then converges to the viscosity solution of a constrained Hamilton-Jacobi equation. 

The hyperbolic scaling and the Hopf-Cole transform make the kinetic equation stiff.
%when reaching the asymptotic regime. 
Thus, the numerical resolution of the problem is challenging, since the standard numerical methods usually lead to high computational costs in these regimes. The \emph{Asymptotic Preserving} (AP) schemes have been typically introduced to deal with this difficulty, since they are designed to be stable along the transition to the macroscopic regime. 
The scheme we propose  is adapted to the non-linearity of the problem, enjoys a discrete maximum principle and solves the limit equation in the sense of viscosity.
It  is based on a dedicated micro-macro decomposition, attached to the Hopf-Cole transform. As it is well adapted to the singular limit, our scheme is able to cope with singular behaviors in space (sharp interface), and possibly in velocity (concentration in the velocity distribution). 
Various numerical tests are proposed, to illustrate the properties and the efficiency of our scheme.
%for this problem.
 \end{abstract}
% 
% \tableofcontents

\section{Introduction}

% We are interested in designing a numerical scheme for a stiff and non-linear kinetic equation, which is able to catch its large stiffness asymptotic behavior. The model we consider has been studied in \cite{BouinCalvez} for the linear case $r=0$ and \cite{Bouin} for the non-linear case $r>0$. It reads 

We are interested in designing a numerical scheme for a non-linear kinetic equation in the asymptotic regime.
The model we consider is a non-linear transport-reaction equation
\begin{equation}
 \label{KinEq2}
\partial_t f(t,x,v) +v\cdot \nabla_x f(t,x,v)=\rho(t,x) M(v)-f(t,x,v) + r\rho(t,x) (M(v)-f(t,x,v)),
\end{equation}
with $r\ge 0$,  supplemented with an initial data $f(0,x,v)=\rho_\tin(x)M(v)$. Such models have been introduced in \cite{Schwetlick,Hadeler,CuestaHittmeirSchmeiser}.
The asymptotic regime of \eqref{KinEq2} have been studied in \cite{BouinCalvez, Bouin, Caillerie}, both in the linear case $r=0$, and in the non-linear case $r>0$.
%and we are intested in designing a numerical scheme for \eqref{KinEq2} in the asymptotic regime. 
%with $r\ge 0$, and it is supplemented with an initial data $f(0,x,v)=\rho_\tin(x)M(v)$. It is a transport-reaction equation
%, as introduced in  \cite{Schwetlick,Hadeler,CuestaHittmeirSchmeiser}.
%%, and it models the \emph{run-and-tumble} behavior of bacteria such as \emph{E. coli} a the mesoscopic scale. 
In \eqref{KinEq2}, the distribution function $f$, which depends on $t>0$, $x\in\R^d$, and $v\in V$, where $V$ is a bounded symmetric set of $\R^d$, represents the density of particles at time $t$, at the position $x$, and with  velocity $v$.
The macroscopic density of particles is defined by,
\[
 \rho(t,x)= \lla f\rra:=\int_{v\in V} f(t,x,v)\dd v, \;\; t\ge 0, \;x\in\R^d.
\]
Note that the brackets 
$\lla \cdot\rra$ denote velocity average throughout the paper.
%will always denote  integrals over $V$ in what follows.
%The linear part of \eqref{KinEq2} models a velocity jump process  of bacteria such as \emph{E. coli} at the mesoscopic scale.
For $r=0$, equation \eqref{KinEq2} describes the evolution of the density of particles moving according to a velocity-jump process. 
Indeed, the motion of a particle is composed of phases of free transport, the \emph{run} phases, with a velocity $v$, and of \emph{tumble} phases
in which the particle 
changes velocity instantaneously. The post-tumbling velocity is chosen randomly,
%chooses a new velocity,
according to a given probability density $M$.
We assume that $M$ is even, nonnegative, and continuous. Moreover, it satisfies
%with a probability $M(v)$, where $M$ is a given equilibrium distribution function. It is even, positive, and normalized such that 
\begin{equation}
 \label{Mv}
 \lla M\rra= 1, \;\; \lla vM\rra=0,
\end{equation}
and we will suppose that 
\begin{equation}
\label{Mv2}
 %\exists \delta >0, \forall v\in V, M(v)\ge \delta.
 \inf\limits_{v\in V}M(v)>0.
\end{equation}
Equation \eqref{KinEq2} is complemented with a reaction term in the case $r>0$. It takes into account creation of new particles at rate $r$, and local quadratic saturation. Initial velocity of new particles is drawn randomly from $M$. Averaging with respect to velocity leads to the classical logistic growth $r\rho(1-\rho)$.  
We consider the kinetic equation \eqref{KinEq2}  under an hyperbolic scaling $(t,x,v)\longmapsto (t/\ep,x/\ep,v)$, such that the time and space variables are of same magnitude (see  \cite{BouinCalvezNadin}).
%, to allow the transport phenomena to arise in the equation. 
The kinetic equation \eqref{KinEq2} then reads 
\begin{equation}
\label{KinEq2bis}
 \partial_t f^\ep(t,x,v) + v\cdot \nabla_x f^\ep(t,x,v)=\frac{1}{\ep}\left( \rho^\ep(t,x) M(v) -f^\ep(t,x,v) + r \rho^\ep(t,x,v) (M(v)-f^\ep(t,x,v))\right).
\end{equation}
%With an analogy of the hyperbolic limit of the Fisher-KPP equation, which is a reaction diffusion equation, we make the so-called \emph{WKB ansatz} to study its asymptotic behavior when $\ep$ goes to $0$, see \cite{EvansSouganidis, Freidlin}. It consists in rewriting the distribution function $f$ as
The propogation of fronts for \eqref{KinEq2} has been studied in \cite{BouinCalvezNadin}. To study the asymptotic behavior of \eqref{KinEq2bis} when $\ep$ goes to $0$,
an analogy is made in \cite{BouinCalvez, Bouin} with the sharp front limit of the Fisher-KPP equation. A \emph{WKB ansatz} is introduced, leading to the so-called approximation of geometric optics (see  \cite{EvansSouganidis, Freidlin}). 
%
%\cite{BouinCalvez, Bouin} make an analogy with the Fisher-KPP equation, and make the so-called \emph{WKB ansatz} (see  \cite{EvansSouganidis, Freidlin}). 
It consists in rewriting the distribution function $f^\ep$ as
\begin{equation}
\label{psiep}
 f^\ep=M\e^{-\psi^\ep/\ep}.
\end{equation}
The equation satisfied by  $\psi^\ep$ in the limit  $\ep\to 0$ is then studied. In the case of the kinetic equation \eqref{KinEq2bis}, if 
\[
 0\le f^\ep(0,\cdot,\cdot)\le M,
\]
a maximum principle ensures that $\psi^\ep$ is well defined and remains nonnegative for all $t\ge 0$, see \cite{Bouin}:
\begin{prop}
 Let $r\ge 0$ and 
 %$\psi^\ep\in \mathcal{C}^1_b(\R^+\times \R^d\times V)$, 
 $\psi_\tin\in \mathrm{Lip}(\R^d\times V)$ and bounded.
 Let  $f^\ep=M\e^{-\psi^\ep/\ep}$  a solution of \eqref{KinEq2bis}. Then the phase $\psi^\ep$ is uniformly locally Lipschitz, and the following a priori bound holds
 \begin{equation}
 \label{MaxPrincCont}
  \forall t\ge 0, \;\; 0 \le \psi^\ep(t,\cdot,\cdot)\le \| \psi_\tin\|_\infty.
 \end{equation}
 %where $\psi_\tin=-\ep\ln(\rho_\tin)$.
\end{prop}
In the case of the Fisher-KPP equation, it has been proved that the function $\psi^\ep$ converges to a limit function $\psi^0$, which is the viscosity solution of a Hamilton-Jacobi equation, see \cite{EvansSouganidis, Barles, BarlesEvansSouganidis, Souganidis, CrandallIshiiLions}. Moreover, in the asymptotic regime, the settled population $\rho\sim 1$ is contained in the nullspace of $\psi^0$, see \cite{Evans, BarlesSouganidis, FlemingSouganidis}.

The analysis of propagation phenomena at the mesoscopic scale is motivated by concentration waves of chemotactic bacteria, as observed experimentally \cite{CalvezEColi}. Here, the model under investigation 
does not contain any chemotactic effect, but takes into account cell division. It satisfies the maximum principle, hence it is more amenable for mathematical analysis, following the seminal works by Kolmogorov, Petrovsky, Piskunov \cite{KolmogorovPetrovskyPiskunov}, and Aronson, Weinberger \cite{AronsonWeinberger}. 
%The model was set on a biological basis in \cite{Hadeler} and references therein. 
The first analytical works, where travelling waves are construced, are \cite{Schwetlick}  and \cite{CuestaHittmeirSchmeiser}. Note that the latter develops a micro-macro decomposition to handle the construction of travelling waves near the diffusive regime.
 We also refer to \cite{Hadeler}, and references therein, for a more general presentation of reaction transport equations in biology.

The asymptotic behavior of \eqref{KinEq2bis} in the limit $\ep\to 0$ has been established rigorously in \cite{BouinCalvez, Bouin}. Before stating the main theorem, let us highlight that the formation of fronts if $r>0$ can be understood with very formal considerations on \eqref{KinEq2bis}. Indeed, when $\ep$ goes to $0$,  supposing that the distribution function $f^\ep$ and its density $\rho^\ep$ converge respectively to $f^0$ and $\rho^0$, we have formally at order $0$ in $\ep$
\begin{equation}
\label{intro_etape}
 \rho^0 M - f^0+r\rho^0(M-f^0)=0,
\end{equation}
which, once integrated in $v$,  gives formally an equation for $\rho^0$
\[
 r\rho^0(1-\rho^0)=0.
\]
Thus, the limit density $\rho^0$ is equal either to $0$ or $1$. Going back to \eqref{intro_etape}, in both cases, we obtain that $f^0=\rho^0 M$.
%The function distribution $f^\ep$ converges to a distribution at equilibrium which  equals either  $0$ or $M$, which corresponds to a formation of fronts. 
Formally, when $\ep$ goes to $0$, two areas appear where $f^0$ is equal either to $0$ or $M$. However, the equation for the dynamic of the propagation of the front is lacking.
The \emph{WKB ansatz} is accurate to catch this
phenomenon,
%asymptotic behavior, 
since the nullset of $\psi^0$ corresponds  to the set where $f^0=M$, and $f^0=0$ where $\psi^0$ has positive values. In the one-dimensional case, the equation of the dynamic of the interface, and the asymptotic behavior of $\psi^\ep$ is given by the following theorem, proved in \cite{BouinCalvez, Bouin}:
\begin{thm}
 \label{AsymptBehavior_thm}
 %Under the previous assumptions for $V, M(v)$ and  $r$. 
 Suppose that $V$ is a symmetric and bounded set of $\R$, that $r\ge0$ and that $M$ is a continuous function of $v$ which satisfies \eqref{Mv}-\eqref{Mv2}. Let $f^\ep$ be the solution of \eqref{KinEq2bis}, and
 suppose that the initial data is well prepared
 \[
%f^\ep(0,x,v)=\rho_\tin(x)M(v), \;\;i.\; e.  \;\; \psi_\tin(x)=\psi^\ep(0,x,v)=-\ep\ln(\rho_\tin),
\psi(0,x,v)=\psi_\tin(x), \;\; i.\;e\;\; f^\ep(0,x,v)=\e^{-\psi_\tin(x)/\ep}M(v), \;\;\text{and}\;\; \rho^\ep(0,x)=\e^{-\psi_\tin(x)/\ep},
 \]
 then $(\psi^\ep)_\ep$ converges locally uniformly towards $\psi^0$, where $\psi^0$ does not depend on $v$. Moreover, $\psi^0$ is the unique viscosity solution of one of the following Hamilton-Jacobi equations:
 \begin{enumerate}
  \item If $r=0$, then $\psi^0$ solves the standard Hamilton-Jacobi problem
  \begin{equation}
   \label{HJr0}
   \left\{
   \begin{array}{l l}
    \partial_t \psi^0 + H\left(\partial_x\psi^0\right) =0, & \text{\;on\;} \R^*_+\times\R, \\ 
    \psi^0(0,\cdot)=\psi_\tin(\cdot). &
   \end{array}
   \right.
  \end{equation}
\item If $r>0$, then the limit equation is the following constrained Hamilton-Jacobi equation
\begin{equation}
 \label{HJnl}
 \left\{ 
 \begin{array}{l l}
  \min\left\{\psi^0, \partial_t \psi^0+H\left(\nabla_x\psi^0\right)+r\right\}=0, & \text{\;on\;}\R^*_+\times \R, \\
  \psi^0(0,\cdot)=\psi_{\tin}(\cdot). & 
 \end{array}
 \right.
\end{equation}
 \end{enumerate}
where, in both cases, the Hamiltonian $H(p)$ is  defined implicitly by 
\begin{equation}
\label{Hnl}
 \lla \frac{M}{1+r+H(p)-v p}\rra =\frac{1}{1+r}.
\end{equation}
\end{thm}
%If $d\ge 2$ or if \eqref{Mv2} is not satisfied, the equation satisfied by $\psi^\ep$ in the limit $\ep\to 0$ is still the Hamilton-Jacobi equation \eqref{HJr0}-\eqref{HJnl}, but a constraint is added in the definition of $H$, to ensure the denominator of \eqref{Hnl} to be nonnegative (see \cite{Caillerie}). 
If $d\geq2$ or if \eqref{Mv2} is not satisfied, it may happen that the implicit definition for the Hamiltonian $H(p)$ \eqref{Hnl} is inconsistent for large values of $p$. The definition of the Hamiltonian has been extended in \cite{Caillerie} to deal with large $p$ in full generality. We refer to Section \ref{SecScheme_fin} for the details. Numerical results are presented in order to illustrate the singular behaviours which arise in this case (concentration in the velocity variable). We also refer to the recent work in \cite{BouinCaillerie} for propagation of front in higher dimension. 
%A numerical test is proposed in section \ref{SecScheme_fin} to highlight this behavior.

The goal of this paper is to construct a numerical scheme able to compute the solution of \eqref{KinEq2bis} in all  regimes of $\ep$, in the one-dimensional case. 
Indeed, the fast relaxation towards the equilibrium distribution function $M$ and the outbreak of a sharp interface in the space variable, make the rescaled kinetic equation \eqref{KinEq2bis} stiff for small $\ep$.
%Indeed, for small $\ep$, the rescaled kinetic equation \eqref{KinEq2bis} becomes stiff. 
%As a consequence, when they are used for this kind of problems,  
Standard numerical methods for partial differential equations require in general the use of refined grids, when they do not take into account the special structure of the problem.
%typically because a relation linking $\ep$ and the discretization parameters has to be fulfilled. 
% In the case of \eqref{KinEq2bis}, this comes from the formation of a stiff front, which has width of order $\ep$. To catch this front, it is indeed necessary to choose a refined grid in the asymptotic regime when the standard numerical methods are used.
% It leads to an increase of the computational cost, and makes the asymptotic regime 
% barely reachable
% %unatteignable 
% in practice. 
%Moreover, numerical diffusion phenomena often arise for small values of $\ep$ regimes, which makes the asymptotics even harder to capture.  
\emph{Asymptotic Preserving} (AP) schemes have been introduced  to avoid 
%these problems,
the problems arising with standard numerical methods when solving stiff asymptotic problems,
see \cite{Jin, Klar1, Klar2}. Indeed, they are constructed exactly with the purpose of  being  stable along the transition from the mesoscopic regime ($\ep\sim 1$) to the macroscopic regime ($\ep\ll 1$). Their properties are often summarized by the following diagram
\[
\begin{array}{c c c}
 P_\ep &  \xrightarrow[{}]{~~~~\ep\to 0~~~~}  & P_0 \vspace{4pt}   \\
\rotatebox{90}{$\xrightarrow[{}]{~~~\dt\to 0~~~}$}   
& &\rotatebox{90}{$\xrightarrow[{}]{~~~\dt\to 0~~~}$} 
\\
 S_\ep^{\dt} & \xrightarrow[{}]{~~~~\ep\to 0~~~~} & S_0^{\dt} 
\end{array}
\]
It can be understood as follows: considering an $\ep$-dependent problem $P_\ep$ which converges to a limit problem $P_0$ when $\ep$ goes to $0$ (see Th. \ref{AsymptBehavior_thm}), the AP scheme $S_\ep^{\dt}$ must be consistent with $P_\ep$ when $\ep$ is fixed. In addition, when the discretization parameters $\dt$ are fixed, it has to converge when $\ep$ goes to $0$ to a limit scheme $S_0^{\dt}$, which is required to be consistent with $P_0$. An AP scheme can also enjoy the stronger property of being \emph{Uniformly Accurate} (UA), which means that its accuracy does not depend on $\ep$.

The development of AP schemes for stiff kinetic equations is necessary to ensure 
%they can be numerically solved for all the 
their numerical resolution in all
regimes of $\ep$. Many AP strategies have been proposed for various asymptotics of linear kinetic equation, see for instance \cite{CarilloGoudonLafitteVecil, JinPareschi1, JinPareschiToscani, BuetCordier, NaldiPareschi, LemouMieussens, LemouMehats, BennouneLemouMieussens}. 

For the particular asymptotics we are considering, an asymptotic scheme has been proposed in \cite{LuoPayneAsympt}, for the linear case $r=0$ in \eqref{KinEq2bis}, complemented with an efficient scheme for the limit equation \eqref{HJr0} in \cite{LuoPayneEikonal}. Asymptotic schemes are based on expansions of the solution of the equation in  formal power series of $\ep$. Contrary to AP schemes, they are designed to approach the asymptotic behavior of the equation, but are not relevant for the mesoscopic regimes, where $\ep\sim 1$. Indeed, the formal power series in $\ep$ are truncated, which restrains  accuracy of the scheme to the small $\ep$ regimes. 
Moreover, the scheme proposed in \cite{LuoPayneAsympt} is not adapted to the non-linear case $r>0$ in \eqref{KinEq2bis}, and it is not clear that the approach they use can be easily adapted to the non-linear case. Conversely, AP schemes are designed to be efficient in all regimes of $\ep$. To do so, it is ensured that no approximation in $\ep$ is done when constructing the scheme. As a consequence, the scheme is consistent with the equation when any $\ep>0$ is fixed.  In addition, since they are  constructed to be efficient in the asymptotic regime, they catch perfectly the asymptotic behavior of the equation. Note also that, contrary to the asymptotic schemes, the UA property provides 
%the asymptotic behavior 
in addition the correct behavior 
of the numerical solutions in the intermediate regime, with no constraint on the parameter $\ep$. As the scheme proposed in this paper enjoys the AP property even in the non-linear case $r>0$ in \eqref{KinEq2bis}, and since the numerical tests suggest that it also enjoys the UA property, it is accurate in all regimes of $\ep$.

We propose here a scheme for
%$\psi^\ep$,  the Hopf-Cole transform of \eqref{KinEq2bis} defined in \eqref{psiep}, 
\eqref{KinEq2bis} in the one-dimensional case,
which enjoys the AP property for the limit equations \eqref{HJr0}-\eqref{HJnl}. It is based on the  following reformulation of \eqref{KinEq2bis} 
\begin{equation}
 \label{KinEqpsi}
 \partial_t \psi^\ep +v\partial_x \psi^\ep = 1+r\lla M \e^{-\psi^\ep/\ep}\rra -(1+r)\lla M\e^{-\psi^\ep/\ep}\rra \e^{\psi^\ep/\ep},
\end{equation}
where $\psi^\ep$ is the Hopf-Cole transform of $f^\ep$, defined in \eqref{psiep}. This problem is still stiff when $\ep$ goes to $0$, and its non-linear character must be carefully dealt with.
%, to design an AP scheme for it. 
Indeed, it is necessary to ensure that the computational cost and the accuracy of the resolution of the non-linear system
%, that arises in the description of the scheme, 
are constant with respect to $\ep$. Another difficulty we have to consider, is the fact that the solution of \eqref{KinEqpsi} converges when $\ep$ goes to $0$ to the viscosity solution of one of the Hamilton-Jacobi equations \eqref{HJr0}-\eqref{HJnl}. The discretization of the kinetic equation \eqref{KinEq2bis} has to be carefully discussed, since it yields the discretization of the limit equation, and that it is required to catch the viscosity solution of the Hamilton-Jacobi equation. Moreover, it must be robust enough to deal with the possible lack of regularity of the solutions of \eqref{KinEq2bis} in the asymptotic regime.
%, especially in the case of a shock. 
Indeed, the solution is expected to be only locally Lipschitz regular, and may present some $\mathcal{C}^1$ discontinuities.
Eventually, when $r>0$ the limit equation \eqref{HJnl} is a constrained problem, thus the AP scheme will be designed to respect this constraint.

The scheme we proposed is based on an adaptation of micro-macro schemes for linear kinetic equations (see \cite{LemouMehats, LemouMieussens, LemouRelaxedMM}). The micro-macro decomposition is modified to be compatible with the nonlinearity of \eqref{KinEqpsi}. In addition, the resolution of the nonlinear system, that is needed to compute the numerical solution of \eqref{KinEqpsi}, is carefully discussed to ensure the computational cost of the method does not depend on $\ep$. The scheme enjoys the AP property, and the limit scheme catches the viscosity solution of the limit equation. It also satisfies a maximum principle, which is a discrete equivalent of \eqref{MaxPrincCont}.

The paper is organized as follows:  the scheme and a formal derivation of the asymptotic behavior of \eqref{KinEqpsi}, on which  the construction of the scheme is based,  are presented in the next section. Then, the numerical resolution of the non-linear system is discussed in Section \ref{SecNLS}, and the discrete maximum principle is proved in Section \ref{SecMP}. The proof of the AP character of the scheme is completed in Section \ref{SecAP}. Eventually, various numerical tests are proposed in Section \ref{SecTests}, to highlight the properties of the scheme.
%However, in the case of the non-linear kinetic equation \eqref{KinEq2bis} we are considering, only an asymptotic scheme have been proposed in \cite{LuoPayneAsympt} , completed with an efficient scheme for the limit equation \cite{LuoPayneEikonal}. 

\vspace{10pt}
\textbf{Acknowledgements.} 
The author wishes to thank Vincent Calvez for having suggested this problem to her, and for very interesting discussions about it. She would also like to thank Emeric Bouin for his kind help. 
\\This project has received funding from the European Research Council (ERC) under the European Union's Horizon $2020$ research and innovation programme (ERC starting grant MESOPROBIO \no $639638$).

%In the case of the Fisher-KPP equation, it has been proved that in the asymptotic regime the population is contained in the nullspace of 

% We are interested in numerical schemes able to catch propagation phenomena in kinetic models. The linear case of the model we consider in this paper has been studied in \cite{BouinCalvez}, and the non-linear case has been treated in \cite{Bouin}.
% Let us recall that this model has been introduced the propagation of bacteria \emph{E. coli} at the mesoscopic scale. T

\section{The numerical scheme}

In this section, we construct an AP scheme for \eqref{KinEq2bis} in the one-dimensional case, using its reformulation with the Hopf-Cole transform of $f^\ep$ \eqref{KinEqpsi}. 
We focus here on the relevant one-dimensional case, as usually done when investigating front propagation. We refer to \cite{BouinCaillerie} for the analysis of front propagation in the higher dimensional case.   
%The one-dimensional case is relevant, since the propagation of bacteria modelled by \eqref{KinEq2bis} has been observed in a 
%one-dimensional channel, 
%thin tube,
%see \cite{CalvezEColi}.
This scheme is based on an adaptation of the micro-macro decomposition proposed in \cite{LemouMieussens, LemouMehats, LemouRelaxedMM} 
for the construction of AP schemes near the diffusive regime of linear kinetic equations
%for linear kinetic equations. 
%It consists 
Indeed, the usual micro-macro decomposition consists
in decomposing the distribution function $f^\ep$ as the sum of its part at equilibrium and of a remainder,
\[
 f^\ep=\rho^\ep M+g^\ep,
\]
where $\lla f^\ep\rra=\rho^\ep$, and $\lla g^\ep\rra=0$. Once injected in a linear kinetic equation, an integration with respect to velocity 
%yields an coupled system of partial differential equations for $\rho$ and $g$.
yields an equation for $\rho^\ep$. It is then substracted from the original kinetic equation, to get a  coupled system of partial differential equations for $\rho^\ep$ and $g^\ep$.
It is equivalent to the original kinetic equation, and the construction of a an AP scheme for the kinetic equation is more explicit with this version. In the case \eqref{KinEqpsi} we are considering, such an approach is not adapted, since the non-linearity of the collision operator prevents the simplifications arising in the linear case. 
Instead, we propose a new micro-macro decomposition of $f^\ep$, in which the density appears multiplied by a remainder. Thanks to the logarithmic transform applied to $f^\ep$, this becomes a linear decomposition for $\psi^\ep$. In the linear case, the micro-macro decomposition is based on a Chapman-Enskog expansion of the distribution function. Here, it makes also appear the limit equation, and it yields in addition an equation for the corrector.
%Instead, we modify the micro-macro decomposition of $f^\ep$, to adapt it to the collision operator we consider. 
Namely, we decompose $\psi^\ep$ defined in \eqref{psiep} as follows
\[
 \psi^\ep(t,x,v)=\phi^\ep(t,x)+\eta^\ep(t,x,v),\;\; (t,x,v)\in \R^*_+\times \R\times V,
\]
where $\phi^\ep$ stands for the Hopf-Cole transform of $\rho^\ep$
\begin{equation}
 \label{phiep}
 \rho^\ep(t,x)=\e^{-\phi^\ep(t,x)/\ep}, \text{\;that is\;} \phi^\ep=-\ep\ln(\rho^\ep),
\end{equation}
and $\eta^\ep$ denotes what remains to reconstruct $f^\ep$
\begin{equation}
 \label{etaep}
 \frac{f^\ep}{\rho^\ep M}=\e^{-\eta^\ep/\ep},
\end{equation}
such that
\begin{equation}
\label{conserved}
 \lla M \e^{-\eta^\ep/\ep}\rra=1.
\end{equation}
Once injected in \eqref{KinEqpsi}, this decomposition and the conserved quantity \eqref{conserved} yield 
\begin{equation}
 \label{KinEq_bis}
 1-\partial_t (\phi^\ep + \eta^\ep) - v\partial_x (\phi^\ep +\eta^\ep) +r \e^{-\phi^\ep/\ep}=(1+r) \e^{\eta^\ep/\ep}.
\end{equation}
With this reformulation of \eqref{KinEqpsi}, it is easy to highlight formally the asymptotic behavior of the model, presented in Theorem \ref{AsymptBehavior_thm}. Moreover, the formal computations we present here are also the basis of the construction of the scheme we propose. Theorem \ref{AsymptBehavior_thm} states that when $\ep$ goes to $0$, the distribution function $f^\ep$ converges to a distribution at equilibrium, since $\psi^\ep$ converges to a phase $\psi^0$ which does not depend on $v$, and which solves one of the limit equations \eqref{HJr0}-\eqref{HJnl}. The asymptotic behavior of $\phi^\ep$ and $\eta^\ep$ is formally, 
\begin{equation}
 \label{asymptbehavior_phieta}
 %\phi^\ep\underset{\ep\to 0}\longrightarrow\phi^0, \;\;\;\; \eta^\ep\underset{\ep\to 0}\longrightarrow 0, 
 \lim\limits_{\ep\to 0}\phi^\ep=\psi^0, \;\;\;\; \lim\limits_{\ep\to 0}\eta^\ep=0,
\end{equation}
since $\phi^\ep$ is the Hopf-Cole transform of $\rho^\ep$.
Moreover, the equation safisfied by $\psi^0$ appears naturally from \eqref{KinEq_bis}. Indeed, it can be reformulated as follows
\[
 \frac{M}{1-\partial_t (\phi^\ep +\eta^\ep) -v\partial_x \left( \phi^\ep +\eta^\ep\right) + r\e^{-\phi^\ep/\ep}}=\frac{M\e^{-\eta^\ep/\ep}}{1+r},
\]
and it becomes, after an integration with respect to velocity
\begin{equation}
\label{1_integrated}
 \lla\frac{M}{1-\partial_t (\phi^\ep +\eta^\ep) -v\partial_x \left( \phi^\ep +\eta^\ep\right) + r\e^{-\phi^\ep/\ep}}\rra=\frac{1}{1+r}.
\end{equation}
If we replace $\partial_t \phi^\ep$ by $-H^\ep-r$, the structure of the limit Hamilton-Jacobi equations \eqref{HJr0}-\eqref{HJnl} appears. 
%To make the structure of the Hamilton-Jacobi equations \eqref{HJr0}-\eqref{HJnl} appear, we 
Indeed, we define $H^\ep$ by
\[
 \partial_t \phi^\ep+H^\ep+r=0,
\]
and it can be  injected in \eqref{1_integrated}. 
Thanks to \eqref{asymptbehavior_phieta}, if $r=0$ the previous expression formally goes to \eqref{HJr0} when $\ep$ goes to $0$, and if $\phi^0>0$, we obtain \eqref{HJnl} when $\ep$ goes to $0$. Note that \eqref{KinEq_bis} ensures that the previous equation \eqref{1_integrated} is equivalent to the conservation of the quantity \eqref{conserved}.

This formal asymptotic analysis paves the way for designing an AP scheme for \eqref{KinEq2bis}. Indeed, as for the usual micro-macro decomposition, the expressions \eqref{KinEq_bis} and \eqref{1_integrated} provide a coupled system of equations for $\phi^\ep$ and $\eta^\ep$, on which the asymptotic equation appears clearly 
\begin{equation}
 \label{start2}
 \left\{
 \begin{array}{l}
  \displaystyle\partial_t \phi^\ep+H^\ep+r=0 \vspace{4pt}
  \\
 \displaystyle 1-\partial_t (\phi^\ep + \eta^\ep) - v\partial_x (\phi^\ep +\eta^\ep) +r \e^{-\phi^\ep/\ep}=(1+r) \e^{\eta^\ep/\ep} \vspace{4pt}
  \\
 \displaystyle \lla\frac{M}{1-\partial_t \eta^\ep +H^ \ep +r -v\partial_x \left( \phi^\ep +\eta^\ep\right) + r\e^{-\phi^\ep/\ep}}\rra=\frac{1}{1+r}.
 \end{array}
 \right.
\end{equation}
To simplify the expression of the scheme, we will  rather remplace the last line by \eqref{conserved}.

The scheme for \eqref{start2} is written using a symmetric grid for the space variable
\begin{equation}
\label{Discret_x}
 x_i=-x_{\text{max}}+\frac{\dx}{2}+(i-1)\dx, \;\;i=1,\dots, N_x=2 N_x', 
\end{equation}
where $\dx=x_{\text{max}}/N_x'$, and for the velocity variable
\begin{equation}
\label{Discret_v}
 v_j=-v_{\text{max}}+\frac{\dv}{2}+(j-1)\dv, \;\;j=1,\dots, N_v=2 N_v', 
\end{equation}
with $\dv=v_{\text{max}}/N_v'$. The integrations in $v$ will be denoted by $\lla \cdot\rra_{N_v}$ and performed with a simple quadrature rule
\begin{equation}
\label{quadrature}
 \lla f\rra_{N_v}=\dv\sum\limits_{j=1}^{N_v} f(v_j).
\end{equation}
In order to write an upwind scheme for the transport in the $x$ variable, we define
\[
 v_j^+=\max(v_j,0), \;\; v_j^-=\min(v_j,0). 
\]
For the time variable, we will denote by $T$ the final time, $\dt=T/N_t$ the time step, and
\begin{equation}
\label{Discret_t}
 t_n=n\dt,\; n=0,\dots, N_t. 
\end{equation}
As in \cite{Bouin, BouinCalvez}, the initial data for $f^\ep$ will be considered at equilibrium $f(0,\cdot,\cdot)=M\rho_{\text{in}}^\ep=M\e^{-{\phi_{\text{in}}}/{\ep}}$, such that $\phi_{\text{in}}$ is a given function and $\eta^\ep(0,\cdot,\cdot)=0$. Denoting $\eta^n_{i,j}$ (resp. $\phi^n_i$) an approximation of $\eta^\ep(t_n,x_i,v_j)$ (resp. $\phi^\ep(t_n,x_i)$) for all $n\ge 1, i\in\ccl 1,N_x\ccr, j\in\ccl 1,N_v\ccr$, we propose the following scheme for $\phi^\ep$ and $\eta^\ep$
\begin{equation}
 \label{Schemenl}
 \left\{
 \begin{array}{l}
  \displaystyle \frac{\phi^{n+1}_i-\phi^n_i}{\dt}+H^{n+1}_i+r=0,   \vspace{4pt} \\
  \displaystyle 1+H^{n+1}_i+r -\frac{\eta^{n+1}_{i,j}-\eta^n_{i,j}}{\dt}-\left[ v\partial_x (\phi+\eta) \right]^n_{i,j} + r\e^{(r\dt -\phi^n_i)/\ep}\e^{\dt H^{n+1}_i/\ep}=(1+r)\e^{\eta^{n+1}_{i,j}/\ep} \vspace{4pt} \\
  \lla M \e^{-\eta^{n+1}_{i,j}/\ep} \rra_{N_v} =1,
 \end{array}
 \right.
\end{equation}
with the initialization $ \phi^0_i=\phi_{in}(x_i), \;\eta^{0}_{i,j}=0$, and where the transport part is computed with an upwind scheme 
\begin{equation}
\label{upwind}
  [v\partial_{x}(\phi+\eta)]^n_{i,j}
  =v_j^+\left(\frac{\phi^n_i-\phi^n_{i-1}}{\dx}+\frac{\eta^n_{i,j}-\eta^n_{i-1,j}}{\dx} \right)
  +v_j^-\left(\frac{\phi^n_{i+1}-\phi^n_{i}}{\dx}+\frac{\eta^n_{i+1,j}-\eta^n_{i,j}}{\dx} \right).
\end{equation}
Note that the dependence in $\ep$ of $\phi^n_i$ and $\eta^n_{i,j}$ has been omitted to simplify the notations.
To write this scheme, we implicited the stiff terms of the equation, to avoid the problems that may arise in the numerical resolution. From the maximum principle, $\phi^\ep$ is nonnegative in \eqref{start2}. Altough $\e^{-\phi^\ep/\ep}$ is then not stiff, we decided to implicit it in the scheme. Indeed, with this formulation, a discrete maximum principle can be proved for the scheme (see Section \ref{SecMP}), and we observed that it is not satisfied in the numerical tests if $\e^{-\phi^ \ep/\ep}$ is taken explicit. Eventually, the use of an upwind scheme \eqref{upwind} for the transport part is important, since it provides the monotonicity properties needed to prove that the limit scheme catches the viscosity solution of the asymptotic model \eqref{HJr0}-\eqref{HJnl}.

In practise, to compute $\left(\phi^{n+1}_i,\eta^{n+1}_{i,j}\right)_{(i,j)\in\ccl 1,N_x\ccr\times \ccl 1,N_v\ccr}$, one must solve the non-linear system 
\begin{equation}
 \label{DefetaHnl}
 \left\{
 \begin{array}{l}
   \displaystyle 1+H^{n+1}_i+r -\frac{\eta^{n+1}_{i,j}-\eta^n_{i,j}}{\dt}-\left[ v\partial_x (\phi+\eta) \right]^n_{i,j} + r\e^{(r\dt -\phi^n_i)/\ep}\e^{ H^{n+1}_i\dt/\ep}=(1+r)\e^{\eta^{n+1}_{i,j}/\ep} \vspace{4pt} \\
  \lla M \e^{-\eta^{n+1}_{i,j}/\ep} \rra_{N_v} =1,
 \end{array}
 \right.
\end{equation}
in which the unknowns are $\left(H^{n+1}_i,\eta^{n+1}_{i,j}\right)_{(i,j)\in\ccl 1,N_x\ccr \times \ccl 1,N_v\ccr}$.
To ensure the AP property of the scheme, the computational cost of the numerical resolution of this system must not depend on $\ep$. A Newton's method is used for its resolution, it is discussed in the next section. Eventually, we have the following proposition:
\begin{prop}
 \label{prop_scheme}
 Consider the scheme \eqref{Schemenl} defined for all $i\in\ccl 1,N_x\ccr, j\in\ccl 1, N_v\ccr, n\in \ccl 0, N_t-1\ccr$, and suppose that the accuracy and computational cost of the resolution of the non-linear problem \eqref{DefetaHnl} are independent of $\ep$. This scheme has the following properties:
 \begin{enumerate}
  \item The scheme is of order $1$ for any fixed $\ep>0$.
  \item The scheme enjoys a discrete maximum principle : Let $m>0$ such that  $0\le\phi^0_i\le m$, and suppose that $\eta^0_{i,j}=0$ for all $(i,j)\in\ccl 1,N_x\ccr\times \ccl 1,N_v\ccr$. Then, the following bounds hold: $\forall n\in\ccl 0,N_t\ccr, \forall (i,j)\in \ccl 1,N_x\ccr\times \ccl 1,N_v\ccr,$
 \begin{align*}
   & 0\le\phi^n_i\le m \\ & 0\le \phi^n_i+\eta^n_{i,j}\le m.
 \end{align*}
\item The scheme is AP: 
%when the discretization parameters are fixed, the scheme converges to a scheme which catches the viscosity solution of the asymptotic equation \eqref{HJr0}-\eqref{HJnl}:
when the discretization parameters are fixed, the scheme tends when $\ep$ goes to $0$ to one of the following scheme:
\begin{align}
 \text{If\;\;} r=0:\;\; & \frac{\phi^{n+1}-\phi^n}{\dt}+H^{n+1}_i=0, \;\; \phi_i^0=\phi_{\text{in}},  \label{LimSchemer0} \\
 \text{If\;\;} r>0:\;\; & \min\left( \frac{\phi^{n+1}_i-\phi^{n}_i}{\dt}+H^{n+1}_i +r, \phi^{n+1}_i \right)=0, \;\; \phi_i^0=\phi_{\text{in}},  \label{LimSchemenl}
\end{align}
with $\phi^n_i\ge 0$ for all $(n,i)\in \ccl 1,N_t\ccr\times\ccl 1,N_x\ccr$, and where in both cases $H^{n+1}_i$ is defined by
\begin{equation}
 \label{LimDefHnl}
 \lla \frac{M}{1+H^{n+1}_i+r-\left[ v\partial_x \phi \right]^n_{i,j}} \rra_{N_v}=\frac{1}{1+r}.
\end{equation}
These two schemes catch respectively
%The two schemes are consistent with
the viscosity solutions of the asymptotic problems \eqref{HJr0}-\eqref{HJnl}, when the discretization parameters tend to zero.
 \end{enumerate}
\end{prop}
\begin{rmq}
 The numerical tests we propose in  Section \ref{Sec_Scheme_UA} suggest that the scheme enjoys the UA property.
\end{rmq}
The consistency and the order of the scheme for fixed $\ep$ are clear, since we only used finite differences to write it. The second point of this proposition is proved in Section \ref{SecMP} and the AP character of the scheme is proved in Section \ref{SecAP}. The properties claimed in this proposition are illustrated by various numerical tests in Section \ref{SecTests}.

% \begin{equation}
%  \label{LimSchemer0}
%  %\left\{
%  \begin{array}{l}
%   \frac{\phi^{n+1}-\phi^n}{\dt}=-H^{n+1}_i, \;\; \phi_i^0=\phi_{\text{in}}(x_i),
%  \end{array}
%  %\right.
% \end{equation}
% if $r=0$, and
% \begin{equation}
%  \label{LimSchemenl}
%  %\left\{
%  \begin{array}{l}
%   \min\left( \frac{\phi^{n+1}_i-\phi^{n}_i}{\dt}+H^{n+1}_i +r, \phi^{n+1}_i \right)=0, \;\; \phi_i^0=\phi_{\text{in}}(x_i),
%  \end{array}
%  %\right.
% \end{equation}

\section{Resolution of the non-linear system}
\label{SecNLS}

The scheme \eqref{Schemenl} proposed in the previous section is both non-linear and
%implicitly defined. 
implicit.
In practice, being given the values $\phi^n_i,\; \eta^n_{i,j}$, the numerical resolution of the non-linear system \eqref{DefetaHnl} is needed to compute $\phi^{n+1}_i,\;\eta^{n+1}_{i,j}$. Furthermore, since the system \eqref{DefetaHnl} depends on $\ep$, and contains  some stiff terms, its numerical resolution may become costly when $\ep$ goes to $0$. This would break the AP property, as the computational cost is expected not to increase when $\ep$ decreases. In this section, we explain how 
the Newton method can be tuned to resolve this issue,
%a Newton's method can be used to solve \eqref{DefetaHnl}, 
and how it can be implemented to make the computational cost independent on the stiffness of the problem.

The system \eqref{DefetaHnl} contains $N_x(N_v+1)$ unknowns, for the same number of equations. %As the direct resolution may become costly, 
We propose 
%in this section
%a method 
to consider it as $N_x$ systems of $N_v+1$ unknowns, each of them being numerically solved with  Newton's method. Indeed, if $i_0\in\ccl1,N_x\ccr$ is fixed, the system 
\begin{equation}
 \label{GHstartnl}
 \left\{
 \begin{array}{l}
   \displaystyle 1+H^{n+1}_{i_0}+r -\frac{\eta^{n+1}_{{i_0},j}-\eta^n_{{i_0},j}}{\dt}-\left[ v\partial_x (\phi+\eta) \right]^n_{i_0,j} + r\e^{(r\dt -\phi^n_{i_0})/\ep}\e^{\dt H^{n+1}_{i_0}/\ep}=(1+r)\e^{\eta^{n+1}_{i_0,j}/\ep} \vspace{4pt} \\
  \lla M \e^{-\eta^{n+1}_{i_0,j}/\ep} \rra_{N_v} =1,
 \end{array}
 \right.
\end{equation}
where the unknowns are $\left(\left(\eta^{n+1}_{i_0,j}\right)_{j\in\ccl 1,N_v\ccr},H^{n+1}_{i_0}\right)$, is a closed system of $N_v+1$ unknowns and $N_v+1$ equations, which can be reformulated as follows 
\[
 F\left(\left(\eta^{n+1}_{i_0,j}\right)_{j\in\ccl1,N_v\ccr},H^{n+1}_{i_0}\right)=0,
\]
with $F:\R^{N_v+1}\longrightarrow\R^{N_v+1}$. The Jacobian matrix of $F$ can be computed explicitly to apply Newton's method. It writes
\begin{equation}
 \label{JacFnl}
 DF\left(\left(\eta^{n+1}_{i_0,j}\right)_{j\in\ccl1,N_v\ccr},H^{n+1}_{i_0}\right)=
 \begin{pmatrix}
  \alp_1 & 0 	& \cdots & 0	 & \delta_1\\
  0	 & \alp_2& \ddots	& \vdots & \vdots \\ 
  \vdots & \ddots & \ddots & 0 & \delta_{N_v-1} \\
  0 & \cdots & 0 & \alp_{N_v} & \delta_{N_v} \\
  \gamma_1 & \cdots &\gamma_{N_v-1} & \gamma_{N_v} & 0
 \end{pmatrix},
\end{equation}
%\textcolor{red}{On en est là}
where the coefficients $\alp_j, \gamma_j, \delta_j$ are defined for $j\in\ccl 1,N_v\ccr$ by
%\begin{align*}
 \[
 \alp_j=-\frac{1}{\dt}-\frac{1+r}{\ep}\e^{-\eta^{n+1}_{i,j}/\ep},\;\;
 \delta_j=1+\frac{r\dt}{\ep}\e^{(r\dt  -\phi^n_i)/\ep}\e^{H^{n+1}_i\dt/\ep},\;\;
 \gamma_j= -\frac{\dv}{\ep}M(v_j)\e^{-\eta^{n+1}_{i,j}/\ep}.
 \]
%\end{align*}
Note that their dependence in $i_0$ is omitted to simplify the notations. 
Since the coefficients of \eqref{JacFnl} are stiff when $\ep$ is small, it is more convenient to compute analytically the inverse of this matrix, instead of doing it numerically. Indeed, this can be computed explicitly, and it reads
%The inverse of \eqref{JacFnl} can also be computed explicitely, it reads
\begin{equation}
 \label{InvDFnl}
%\left( DF\left(\left(G_{i_0,j}\right)_{j\in\ccl1,N_v\ccr},H_{i_0}\right)\right)^{-1}=
\frac{1}{S}
\begin{pmatrix}
\displaystyle \frac{S}{\alp_1}-\frac{\delta_1}{\alpha_1}\frac{\gamma_1}{\alp_1} \hspace{2pt}  & 
\displaystyle -\frac{\delta_1}{\alp_1} \frac{\gamma_2}{\alp_2}  \hspace{2pt} & 
\displaystyle \dots \hspace{2pt} & 
\displaystyle - \frac{\delta_1}{\alp_1} \frac{\gamma_{N_v}}{\alp_{N_v}} \hspace{2pt}& 
\displaystyle\frac{\delta_1}{\alp_1} \vspace{6pt}\\
\displaystyle- \frac{\delta_2}{\alp_2} \frac{\gamma_1}{\alp_1} & 
\displaystyle\frac{S}{\alp_2} -\frac{\delta_2}{\alp_2}\frac{\gamma_2}{\alp_2}   & 
\displaystyle\cdots & \displaystyle -\frac{\delta_2}{\alp_2} \frac{\gamma_{N_v}}{\alp_{N_v}} & 
\displaystyle\frac{\delta_2}{\alp_2} \vspace{2pt}\\
& & & & \\
\displaystyle\vdots &\displaystyle \ddots & \displaystyle\ddots  & \displaystyle\vdots & \displaystyle\vdots \vspace{2pt}\\
& & & & \\
 \displaystyle- \frac{\delta_{N_v}}{\alp_{N_v}} \frac{\gamma_1}{\alp_1} & \displaystyle\cdots & \displaystyle-\frac{\delta_{N_v-1}}{\alp_{N_v}} \frac{\gamma_{N_v-1}}{\alp_{N_v-1}} &  \displaystyle\frac{S}{\alp_{N_v}} -     \frac{\delta_{N_v}}{\alp_{N_v}} \frac{\gamma_{N_v}}{\alp_{N_v}} &\displaystyle  \frac{\delta_{N_v}}{\alp_{N_v}} \vspace{6pt}\\
 \displaystyle\frac{\gamma_1}{\alp_1} & \displaystyle\cdots   & \displaystyle\frac{\gamma_{N_v-1}}{\alp_{N_v-1}} & \displaystyle \frac{\gamma_{N_v}}{\alp_{N_v}} & -1
\end{pmatrix},
\end{equation}
where 
\[
 S=\sum\limits_{j=1}^{N_v} \frac{\gamma_j}{\alp_j}\delta_j.
\]
Having an explicit formula for $\left(DF\left(\left(G^{n+1}_{i_0,j}\right)_{j\in\ccl1,N_v\ccr},H^{n+1}_{i_0}\right)\right)^{-1}$ presents two advantages. First of all, it reduces the computational cost by avoiding the numerical inversion of \eqref{JacFnl}, which may be of large dimension. 
One can also notice that the construction of the matrix \eqref{InvDFnl} can be done with a moderate computational cost, since it is mostly made of an assembly of the vectors $\left(\delta_i/\alp_i\right)_i$, $\left(1/\alp_i\right)_i$, and $\left( \gamma_i/\alp_i\right)_i$.
But the main advantage is that the coefficients of \eqref{InvDFnl}, which contain
\[
 \frac{\gamma_j}{\alp_j}=-\frac{\dt\dv M(v_j)\e^{-\eta^{n+1}_{i_0,j}}/\ep}{\ep+(1+r)\dt \e^{\eta^{n+1}_{i_0,j}/\ep}},
\]
and
\[
 \frac{\delta_j}{\alp_j}=-\frac{\ep \dt +r\dt^2 \e^{(r\dt -\phi^n_i)/\ep}\e^{ H^{n+1}_i\dt/\ep}}{\ep+(1+r)\dt\e^{\eta^{n+1}}_{i,j}/\ep},
\]
do not present any singularity  as $\ep$ goes to $0$. Indeed, thanks to the discrete maximum principle that is proved in Section \ref{SecMP},  $\e^{(r\dt -\phi^n_i+ H^{n+1}_i\dt)/\ep}$ is bounded with respect to $\ep$. 
In fact, the discrete maximum principle states that
$\phi^{n+1}_i$ defined by the first line of \eqref{Schemenl} as 
\[
 \phi^{n+1}_i=\phi^n_i-r\dt  - H^{n+1}_i\dt,
\]
remains nonnegative.
The proof of the AP character of the scheme in Section \ref{SecAP}
also states as a preliminary result that
$\e^{{\eta^{n+1}_{i_0,j}}/{\ep}}$ and $\e^{-{\eta^{n+1}_{i_0,j}}/{\ep}}$  are bounded with respect to $\ep$.
% But, the main advantage is 
% %that it appears 
% that, supposing that  $\e^{{\eta^{n+1}_{i_0,j}}/{\ep}}$ 
% %and $\e^{\dt {G^{n+1}_{i_0,j}}/{\ep}}$ are 
% and $\e^{(r\dt -\phi^n_i+ H^{n+1}_i\dt)/\ep}$ are
% bounded in $\ep$,  the coefficients of \eqref{InvDFnl} do not contain singular terms in $\ep$ since 
% % \[
% %  \frac{1}{\alp_j}=-\frac{\ep}{\ep+\dt \e^{{\eta^n_{i_0,j}}/{\ep}}\e^{-\dt {G^{n+1}_{i_0,j}}/{\ep}}},
% % \]
% \[
%  \frac{\gamma_j}{\alp_j}=-\frac{\dt\dv M(v_j)\e^{-\eta^{n+1}_{i_0,j}}/\ep}{\ep+(1+r)\dt \e^{\eta^{n+1}_{i_0,j}/\ep}}
% \]
% and
% % \[
% %  \frac{\gamma_j}{\alp_j}=-\dt\dv\frac{\e^{\dt{G^{n+1}_{i_0,j}}/{\ep}}\e^{-{\eta^n_{i_0,j}}/{\ep}}}{\ep+\dt\e^{-\dt{G^{n+1}_{i_0,j}}/{\ep}}\e^{{\eta^n_{i_0,j}}/{\ep}} }M(v_j).
% % \]
% \[
%  \frac{\delta_j}{\alp_j}=-\frac{\ep \dt +r\dt^2 \e^{(r\dt -\phi^n_i)/\ep}\e^{ H^{n+1}_i\dt/\ep}}{\ep+(1+r)\dt\e^{\eta^{n+1}}_{i,j}/\ep}
% \]
% %As we will see in what follows, 
% As proved in Section \cite{SecAP} and \ref{Sec},
% the quantity $\eta^{n+1}_{i,j}/\ep$ is uniformly bounded in $\ep$, and $\phi^{n+1}_i$ defined by the first line of \eqref{Schemenl} as 
% \[
%  \phi^{n+1}_i=\phi^n_i-r\dt  - H^{n+1}_i\dt,
% \]
% remains nonnegative.
Using \eqref{InvDFnl} in the numerical computations will avoid singularities in the numerical computation of the inverse of \eqref{JacFnl}, that may appear for the small values of $\ep$. %Hence, the difficulty of the  resolution of the system \eqref{GHstartnl} will %be possible for all values of $\ep>0$. 
%not depend on $\ep$. 

As a conclusion, we propose to compute  the solution of \eqref{GHstartnl}  with   the Newton method initialized by
\[
[\eta, H]^0:=
%\left(\left(G^0_j\right)_{j\in\ccl 1,N_v\ccr}, H^0\right)=
\left(\left(\eta^{n}_{i_0,j}\right)_{j\in\ccl 1,N_v\ccr}, H^n_{i_0}\right) 
\]
and with the iterations 
\[
% \left(\left(G^{k+1}_j\right)_{j\in\ccl 1,N_v\ccr}, H^{k+1}\right)=
[\eta, H]^{k+1}=[\eta, H]^k-\left(DF\left([\eta, H]^k\right)\right)^{-1} [\eta, H]^k.
\]
Thanks to the discussion above, the accuracy and the computational cost of the resolution of the system is then independant of $\ep$, which is a necessary condition for the scheme \eqref{Schemenl} to enjoy the AP property.

%\textcolor{red}{\textbf{\`A faire :} Montrer que pour une donnée initiale $GH^0$ assez proche de la solution, il y a convergence de la méthode de Newton. Vu l'expression de $F$ ça doit pouvoir s'écrire (elle est convexe en toutes ses variables). }

%\section{Stability and consistency of the scheme in the non-linear case}
\section{Discrete maximum principle}
\label{SecMP}

The scheme \eqref{Schemenl} has been 
designed including an upwind choice to deal with the transport terms.
%written using an upwind scheme for the transport terms, 
This ensures %its consistency for fixed $\ep$,
%. In addition, for fixed $\ep$ the scheme is stable 
the consistency of the scheme for a given $\ep$, as well as its stability under a CFL condition 
% and its stability
% under a CFL condition. 
\begin{equation}
 \label{CFL}
 v_{\text{max}}\frac{\dt}{\dx}<1.
\end{equation}
The stiff terms in $\ep$ have been treated implicitly to ensure the stability of the scheme when $\ep$ decreases.
Notice that this implies the AP property as well, as we shall see in the next section.
%and its AP property, as we will see in the next section.
%In this section, we prove that the scheme enjoys a maximum property
Moreover, the scheme enjoys a discrete maximum principle, which is the discrete version of the maximum principle  \eqref{MaxPrincCont} proved in \cite{Bouin, BouinCalvez} in the continuous case. In this section, we prove the 
%that is written in 
the following proposition, which contains the second point of Proposition \ref{prop_scheme}.
\begin{prop}
\label{propMP}
Let $m>0$. Assume that for all $  i\in\ccl 1,\; N_x\ccr,   j\in\ccl 1,N_v\ccr, 0\le  \phi^0_i \le m$ and $\eta^0_{i,j}=0$, then for all $ n\in \ccl 0, N_t-1\ccr,$ and for all $(i,j)\in \ccl 1, N_x\ccr\times\ccl1,N_v\ccr$, the following properties hold : 
 %$\forall n\in \ccl 0, N-1\ccr,\; \forall (i,j)\in \ccl 1, N_x\ccr\times\ccl1,N_v\ccr$,
 \begin{enumerate}
  \item $0\le\phi^{n+1}_i\le m$,
  \item $ 0\le \eta^{n+1}_{i,j}+\phi^{n+1}_j\le m$,
  \item $  0  \le  \phi^{n+1}_{i}+\eta^{n+1}_{i,j}-\dt\left(1-\e^{\eta^{n+1}_{i,j}/\ep}\right)+r\dt\e^{\eta^{n+1}_{i,j}/\ep}\left( 1-\e^{-\left( \phi^{n+1}_i+\eta^{n+1}_{i,j} \right)/\ep} \right)\le m $.
 \end{enumerate}
\end{prop}
% 
% \[
% \left( \forall i\in\ccl 1,\; N_x\ccr, 0\le  \phi^0_i \le m\right) \Rightarrow  \left(\forall n\in \ccl 1,\ N\ccr,\; \forall (i,j)\in \ccl 1, N_x\ccr\times\ccl1,N_v\ccr,\; 0\le \eta^n_{i,j}+\phi^n_j\le m  \right),
% \]
%where we recall that we supposed $\eta^0_{i,j}=0$ for all $(i,j)\in \ccl 1,Nx\ccr\times \ccl 1,N_v\ccr$.
\begin{proof} 
 As a preliminary result, let us remark that for a fixed $i\in\ccl 1,N_x\ccr$, there exists at least one index $j_-$ such that $\eta^{n+1}_{i,j_-}\le 0$ and one index $j_+$ such that $\eta^{n+1}_{i,j_+}\ge 0$.  
Otherwise, the equality 
 \[
 \lla M \e^{-\eta^{n+1}_{i,j}/\ep} \rra =1,
\]
can  not be fulfilled, since $\lla M\rra=1$.

The proof of this proposition is done by induction.
 Let us fix $n\in\ccl 0, N_t-1\ccr$ and suppose that for all 
 $(i,j)\in \ccl 1, N_x\ccr\times\ccl1,N_v\ccr$, $0\le \phi^{n}_i+\eta^n_{i,j}\le m$. The third point of the proposition comes by fixing $j\in\ccl 1,N_v\ccr$  such that $v_j\ge 0$ (the demonstration is similar for $v_j< 0$). Indeed, the second line of \eqref{Schemenl} can be recast as follows
\[
 \frac{\phi^{n+1}_{i}-\phi^n_i}{\dt}+\frac{\eta^{n+1}_{i,j}-\eta^n_{i,j}}{\dt}+v_j\frac{\phi^n_i-\phi^n_{i-1}}{\dx}+v_j\frac{\eta^n_{i,j}-\eta^n_{i-1,j}}{\dx}+r\e^{\eta^{n+1}_{i,j}/\ep}=1-\e^{{\eta^{n+1}_{i,j}}/\ep}+r\e^{-\phi^{n+1}_i/\ep},
\]
or, equivalently
\begin{align}
\label{StabilityStartnl}
 \phi^{n+1}_{i}+\eta^{n+1}_{i,j}&-\dt\left(1-\e^{\eta^{n+1}_{i,j}/\ep}\right)+r\dt\e^{\eta^{n+1}_{i,j}/\ep}\left( 1-\e^{-\left( \phi^{n+1}_i+\eta^{n+1}_{i,j} \right)/\ep} \right)\\&=\left( 1-v_j\frac{\dt}{\dx}\right)\left( \phi^n_i+\eta^n_{i,j} \right)+v_j\frac{\dt}{\dx}\left( \phi^n_{i-1}+\eta^n_{i-1,j} \right). \nonumber
\end{align}
If $0\le \phi^n_i+\eta^n_{i,j}\le m$ for all $(i,j)\in\ccl 1,N_x\ccr\times \ccl1,N_v\ccr$, and under the CFL condition
% \[
%  v_{\text{max}}\frac{\dt}{\dx}<1,
% \]
\eqref{CFL},
the right-hand side of \eqref{StabilityStartnl} is a convex combination of quantities lying in the interval $[0,m]$. This yields the inequality
\[0  \le  \phi^{n+1}_{i}+\eta^{n+1}_{i,j}-\dt\left(1-\e^{\eta^{n+1}_{i,j}/\ep}\right)+r\dt\e^{\eta^{n+1}_{i,j}/\ep}\left( 1-\e^{-\left( \phi^{n+1}_i+\eta^{n+1}_{i,j} \right)/\ep} \right)\le m.\]

Let us suppose in addition that for all $i\in \ccl 1,N_x\ccr$, $\phi^n_i\ge 0$. The first point of the proposition is a consequence of the previous inequality. 
We argue by contradiction. Suppose that $\phi^{n+1}_{i_0}< 0$ for some 
%$i_0\in\ccl1,N_x\ccr$, 
$i_0$.
In addition, let $j_0\in\ccl 1,N_v\ccr$ such that $\eta^{n+1}_{i_0,j_0}\le0$. Then the inequality
% 
% Indeed, if we suppose that there exists
% $i_0\in\ccl1,N_x\ccr$ such that $\phi^{n+1}_{i_0}< 0$, and $j_0\in\ccl 1,N_v\ccr$ such that $\eta^{n+1}_{i_0,j_0}\le0$, then the inequality
 \begin{equation}
 \label{etapenl}
  0\le \phi_{i_0}^{n+1} +\eta^{n+1}_{i_0,j_0} -\dt\left(1-\e^{\eta^{n+1}_{i_0,j_0}/\ep}\right)
  +r\dt\e^{\eta^{n+1}_{i_0,j}/\ep}\left( 1-\e^{-\left( \phi^{n+1}_{i_0}+\eta^{n+1}_{i_0,j} \right)/\ep} \right),
 \end{equation}
cannot be fulfilled, since the right-hand side is a sum of negative terms. Hence, for all $j\in\ccl 1,N_v\ccr,\; \eta^{n+1}_{i_0,j}> 0$, but
this is contradictory with the preliminary result stated in the beginning of this proof.
% but the equality
% \[
%  \lla M \e^{-\eta^{n+1}_{i_0,j}/\ep} \rra =1,
% \]
% can then not be fulfilled, since $\lla M\rra=1$.
% then implies that $\forall j\in\ccl 1,N_v\ccr, \eta^{n+1}_{i_0,j}=0$. However, in \eqref{etape}, this equality leads to 
% \[
%  0\le \phi^{n+1}_{i_0}\le m,
% \]
% which contradicts the hypothesis. 
Hence, for all $i\in \ccl 1,N_x\ccr$, $\phi^{n+1}_i\ge 0$. Let  $j_0$ such that $\eta^{n+1}_{i,j_0}\ge 0$. Considering \eqref{etapenl} for $(i,j_0)$, we obtain
\[
 \phi_i^{n+1}\le \phi_i^{n+1}+\eta_{i,j_0}^{n+1} +\dt \left(\e^{\eta_{i,j_0}^{n+1}/\ep}-1\right)+r\dt\e^{\eta^{n+1}_{i,j_0}/\ep}\left( 1-\e^{-\left( \phi^{n+1}_{i}+\eta^{n+1}_{i,j_0} \right)/\ep} \right) \le m,
\]
which gives the upper bound for $\phi_i^{n+1}$, and the first point of the proposition
\[
 0\le \phi^{n+1}_i\le m.
\]

Eventually, the second point of the proposition can be proved. Indeed, if 
$\eta^{n+1}_{i,j}\ge 0$, the inequality $0\le \phi^{n+1}_i+\eta^{n+1}_{i,j}$ follows immediately, and if $\eta^{n+1}_{i,j}<0$, we have
\begin{align*}
 0&\le \phi^{n+1}_i+\eta^{n+1}_{i,j} +\dt\left(\e^{\eta^{n+1}_{i,j}/\ep}-1\right)+r\dt\e^{\eta^{n+1}_{i,j}/\ep}\left( 1-\e^{-\left( \phi^{n+1}_i +\eta^{n+1}_{i,j} \right)/\ep}  \right)\\
 &\le \phi^{n+1}_i+\eta^{n+1}_{i,j} +r\dt \left( 1-\e^{-\left( \phi^{n+1}_i +\eta^{n+1}_{i,j} \right)/\ep}  \right) \\
 & \le \left(1+\frac{r\dt}{\ep}\right)\left( \phi^{n+1}_i+\eta^{n+1}_{i,j} \right),
\end{align*}
which proves the positivity of $\phi^{n+1}_{i}+\eta^{n+1}_{i,j}$. This implies also
\begin{align*}
 \phi^{n+1}_i+\eta^{n+1}_{i,j}+\dt\left(\e^{\eta^{n+1}_{i,j}/\ep}-1\right) &\le \phi^{n+1}_i+\eta^{n+1}_{i,j}+\dt\left(\e^{\eta^{n+1}_{i,j}/\ep}-1\right) +r\dt \e^{\eta^{n+1}_{i,j}/\ep}\left(1-\e^{-\left( \phi^{n+1}_i+\eta^{n+1}_{i,j} \right)/\ep}\right)
 \\&\le m
\end{align*}
Considering the function $f:x\mapsto x+\dt(\e^{x/\ep}-1)$, we remark that since $x/\ep\le \e^{x/\ep}-1$,
%\textcolor{red}{faux, ce n'est vrai que si $\eta^{n+1}_{i,j} \ge 0$ !}
\[
 \eta^{n+1}_{i,j}+\dt (\e^{\eta^{n+1}_{i,j}/\ep}-1)\le m-\phi^{n+1}_i \;\;\Rightarrow\;\; \left(1+\frac{\dt}{\ep}\right)\eta^{n+1}_{i,j}\le m-\phi^{n+1}_{i} \;\;\Rightarrow\;\; \eta^{n+1}_{i,j}\le m-\phi^{n+1}_i, 
\]
because $m-\phi^{n+1}_i\ge 0$.
Eventually, we get the upper bound of $\phi^{n+1}_i+\eta^{n+1}_{i,j}$
\[
 \phi^{n+1}_i+\eta^{n+1}_{i,j}\le m.
\]
\end{proof}

%\section{AP character in the non-linear case}
\section{AP character}
\label{SecAP}

In this section, we prove that when $\ep$ goes to $0$, the quantity $\eta^n_{i,j}$, which depends on $\ep$ is such that $\eta^n_{i,j}/\ep$ is bounded uniformly with respect to $\ep$, and that the limit scheme for \eqref{Schemenl} when $\ep$ goes to $0$ is consistent with \eqref{HJnl}. Let us fix $n\in\ccl 0,N\ccr$, and $(i,j)\in\ccl 1,N_x\ccr\times \ccl 1,N_v\ccr$. Since
\[
 \lla M \e^{-\eta^n_{i,j}/\ep}\rra_{N_v} =1,
\]
where the discrete integration with respect to velocity has been defined in \eqref{quadrature}, and $\lla M\rra_{N_v}=1$, there exists a constant $c$ such that 
\[
  \e^{-\eta^n_{i,j}/\ep}\le c.
\]
Moreover, as
\[
 0\le \phi^n_{i}+\eta^n_{i,j}+\dt\left(\e^{\eta^n_{i,j}/\ep}-1\right)+r\dt \left(\e^{\eta^{n+1}_{i,j}/\ep}-\e^{-\phi^{n+1}_i/\ep}\right) \le m,
\]
with, 
\[
 0\le \phi^n_i+\eta^n_{i,j} \le m, \;\;\text{and}\;\; 0\le \phi^{n+1}_i\le m,
\]
 we obtain that
 \[
   \e^{\eta^{n}_{i,j}/\ep}\le \frac{m+\dt}{\dt}.
 \]
Hence, both $\e^{-\eta^n_{i,j}/\ep}$ and $\e^{\eta^n_{i,j}/\ep}$ are bounded independently of $\ep$, thanks to the maximum principle of the previous section. As a consequence, $\eta^n_{i,j}/\ep$ is bounded independently of $\ep$. In other words, $\eta^{n}_{i,j}$ vanishes when $\ep$ goes to $0$, which is the expected asymptotic behavior of $\eta^\ep$ at the continuous level \eqref{asymptbehavior_phieta}.
To find the limit scheme for \eqref{Schemenl}, it is more convenient to rewrite the third line of \eqref{Schemenl} as
\[
  \lla \frac{ M}{ 1+H^{n+1}_i+r-\frac{\eta^{n+1}_{i,j}-\eta^{n}_{i,j}}{\dt}-\left[v\partial_x(\phi+\eta)\right]^n_{i,j}+r\e^{-\phi^{n+1}_i/\ep}}\rra_{N_v} =\frac{1}{1+r}.
\]
%To obtain this expression, we expressed $\e^{\eta^{n+1}/\ep}$ with the second line of \eqref{Schemenl} and we injected it in the third line. 
These two expressions are equivalent, thanks to the second line of the scheme \eqref{Schemenl}. 
 Hence, when $\ep\to 0$ with fixed $\dt$, \eqref{Schemenl} becomes a scheme for $\phi^n_i$ only that writes
\begin{equation}
 %\label{LimSchemer0}
 %\left\{
 \begin{array}{l}
 \displaystyle \frac{\phi^{n+1}-\phi^n}{\dt}=-H^{n+1}_i, \;\; \phi_i^0=\phi_{\text{in}}(x_i),
 \end{array}
 %\right.
\end{equation}
if $r=0$, and
\begin{equation}
 %\label{LimSchemenl}
 %\left\{
 \begin{array}{l}
 \displaystyle \min\left( \frac{\phi^{n+1}_i-\phi^{n}_i}{\dt}+H^{n+1}_i +r,\; \phi^{n+1}_i \right)=0, \;\; \phi_i^0=\phi_{\text{in}}(x_i),
 \end{array}
 %\right.
\end{equation}
if $r>0$.
In both cases, the maximum principle of the second point of Prop. \ref{prop_scheme} ensures that $\phi^n_i\ge 0$, and  $H^{n+1}_i$ is defined implicitly for all $n\in\ccl 1, N-1\ccr$, and $i\in\ccl 1, N_x\ccr$, by
\begin{equation}
 %\label{LimDefHnl}
 \lla \frac{M}{1+H^{n+1}_i+r-\left[ v\partial_x \phi \right]^n_{i,j}} \rra=\frac{1}{1+r},
\end{equation}
where the space derivative $\left[ v\partial_x \phi \right]^n_{i,j}$ is computed with an upwind scheme as in \eqref{upwind}.

To prove that the scheme \eqref{Schemenl} enjoys the AP property, it remains to show that \eqref{LimSchemer0} and \eqref{LimSchemenl} 
catch the viscosity solution of \eqref{HJr0} and \eqref{HJnl}. To do so, we use a general result on numerical schemes for Hamilton-Jacobi equations, stated in \cite{CrandallLions}. It states that, for a scheme written in the general form 
\begin{equation}
 \label{GeneralSchemenl}
 \phi^{n+1}_i=F(\phi^n_{i-1},\phi^n_i,\phi^n_{i+1}), 
\end{equation}
that can also be written with a \emph{differenced form} 
\begin{equation}
\label{DifferencedFormnl}
 \phi^{n+1}_i=\phi^n_i- h\left(\frac{\phi^n_i-\phi^n_{i-1}}{\dx},\frac{\phi^n_{i+1}-\phi^n_i}{\dx}\right)\dt,
\end{equation}
the following theorem, adapted from \cite{CrandallLions}, holds
\begin{thm}
\label{CrandallLionsThmnl}
 Let $H:\R\longrightarrow\R$ be continuous, $\phi_{\tin}$ be bounded and Lipschitz continuous on $\R$ with Lipschitz constant $L$, and $\phi$ be the viscosity solution of
\begin{equation}
 \label{GeneralHJnl}
  \partial_t\phi+H\left(\partial_x \phi\right)=0,  \;\; \phi(0,x)=\phi_{\text{in}}(x).
\end{equation}
 Under the $CFL$ condition
 \begin{equation}
 \label{CFLnl}
 v_{\text{max}} \frac{\dt}{\dx}\le 1,
 \end{equation}
suppose that the scheme  \eqref{GeneralSchemenl} has a \emph{differenced form} \eqref{DifferencedFormnl}, that the numerical Hamiltonian $h$ is consistent
with \eqref{GeneralHJnl} 
with $H$, 
%\emph{i.\;e.}  $h(p,p)=H(p)$ when $\dv$ tends to $0$,
and locally Lipschitz continuous. Suppose also that the scheme $F$ defined in \eqref{GeneralSchemenl} is a nondecreasing function of each argument. 
Then, there exists a constant $c$ depending only on $\sup{|\phi_{\text{in}}|}$, $L$, $h$ and $T$ such that
\[
 \forall n\in\ccl 1,N\ccr, \forall i\in\ccl 1,N_x\ccr, \;\; \left| \phi^n_i-\phi(t_n,x_i)\right|\le c\sqrt{\dx+\dt}.
\]
\end{thm}
Here, the discretized Hamiltonian $H_{N_v}$ and the numerical Hamiltonian $h$ are implicitly defined. Indeed, if $p\in\R$, $H_{N_v}(p)$ is solution of 
\[
 \lla \frac{M}{1+r+H_{N_v}(p)-vp}\rra_{N_v} =\frac{1}{1+r},
\]
with $1+H_{N_v}(p)-vp>0$, and if $(p,q)\in\R^2$, $h(p,q)$ is solution of 
\[
 \lla \frac{M}{1+r+h(p,q)-v_+ p-v_-q}\rra_{N_v}=\frac{1}{1+r},
\]
with $1+r+h(p,q)-v_+p-v_-q>0$.
Therefore,  the implicit function theorem ensures that $H_{N_v}$ (resp. $h$) is a continuous and locally Lipschitz function of $p$ (resp. $(p,q)$). The consistency of $h$ with $H_{N_v}$ is a consequence of the fact that $h(p,p)=H_{N_v}(p)$. Note that we only consider the discrete integrations $\lla\cdot \rra_{N_v}$ for this result. The fact that the solution of \eqref{HJr0}-\eqref{HJnl} with $H_{N_v}$ converges to the solution of \eqref{HJr0}-\eqref{HJnl} with $H$ is a consequence of stability results for Hamilton-Jacobi equations (see \cite{Barles}).
% except that the numerical integration $\lla \cdot \rra_{N_v}$ is used.
To prove that the schemes \eqref{LimSchemer0}-\eqref{LimSchemenl} catch the viscosity solution of \eqref{HJr0}-\eqref{HJnl} with the integrations in velocity done with $\lla \cdot\rra_{N_v}$, it remains to show that it is nondecreasing in each argument. Considering \eqref{LimSchemer0}-\eqref{LimSchemenl} in the general form \eqref{GeneralSchemenl}, we compute the  partial derivatives of $F$ with the implicit function theorem. They read 
\begin{align*}
 &\frac{\partial F}{\partial \phi^n_{i-1}}\left(\phi^n_{i-1},\phi^n_i,\phi^n_{i+1}\right)
 =\frac{\dt}{\dx}
 \frac{\displaystyle \lla 
 \frac{v_+M}{\left(1+r+h\left(\frac{\phi^n_i-\phi^n_{i-1}}{\dx} ,\frac{\phi^n_{i+1}-\phi^n_{i}}{\dx} \right)-v_+\frac{\phi^n_i-\phi^n_{i-1}}{\dx}-v_-\frac{\phi^n_{i+1}-\phi^n_{i}}{\dx}\right)^2}
 \rra_{N_v}}{
 \displaystyle
 \lla 
 \frac{M}{\left(1+r+h\left(\frac{\phi^n_i-\phi^n_{i-1}}{\dx} ,\frac{\phi^n_{i+1}-\phi^n_{i}}{\dx} \right)-v_+\frac{\phi^n_i-\phi^n_{i-1}}{\dx}-v_-\frac{\phi^n_{i+1}-\phi^n_{i}}{\dx}\right)^2}
 \rra_{N_v}
 } 
 \\
 \vspace{5pt}
 & \frac{\partial F}{\partial \phi^n_{i}}\left(\phi^n_{i-1},\phi^n_i,\phi^n_{i+1}\right)
 =1-\frac{\dt}{\dx}
 \frac{\displaystyle
 \lla
 \frac{(v_+-v_-)M}{\left(1+r+h\left(\frac{\phi^n_i-\phi^n_{i-1}}{\dx} ,\frac{\phi^n_{i+1}-\phi^n_{i}}{\dx} \right)-v_+\frac{\phi^n_i-\phi^n_{i-1}}{\dx}-v_-\frac{\phi^n_{i+1}-\phi^n_{i}}{\dx}\right)^2}
 \rra_{N_v}}{
 \displaystyle
 \lla
 \frac{M}{\left(1+r+h\left(\frac{\phi^n_i-\phi^n_{i-1}}{\dx} ,\frac{\phi^n_{i+1}-\phi^n_{i}}{\dx} \right)-v_+\frac{\phi^n_i-\phi^n_{i-1}}{\dx}-v_-\frac{\phi^n_{i+1}-\phi^n_{i}}{\dx}\right)^2}
 \rra_{N_v}}
 \\
 &\frac{\partial F}{\partial \phi^n_{i+1}}\left(\phi^n_{i-1},\phi^n_i,\phi^n_{i+1}\right)
 =\frac{\dt}{\dx}
 \frac{\displaystyle \lla 
 \frac{\left|v_-\right|M}{\left(1+r+h\left(\frac{\phi^n_i-\phi^n_{i-1}}{\dx} ,\frac{\phi^n_{i+1}-\phi^n_{i}}{\dx} \right)-v_+\frac{\phi^n_i-\phi^n_{i-1}}{\dx}-v_-\frac{\phi^n_{i+1}-\phi^n_{i}}{\dx}\right)^2}
 \rra_{N_v}}{
 \displaystyle
 \lla 
 \frac{M}{\left(1+r+h\left(\frac{\phi^n_i-\phi^n_{i-1}}{\dx} ,\frac{\phi^n_{i+1}-\phi^n_{i}}{\dx} \right)-v_+\frac{\phi^n_i-\phi^n_{i-1}}{\dx}-v_-\frac{\phi^n_{i+1}-\phi^n_{i}}{\dx}\right)^2}
 \rra_{N_v}
 },
\end{align*}
which gives bounds on the partial derivatives of $F$,
\[
\begin{array}{r c l}
 \displaystyle 0& \displaystyle\le \frac{\partial F}{\partial \phi_{i-1}} \le&\displaystyle v_{\text{max}}\frac{\dt}{\dx} \vspace{4pt}
 \\ 
 \displaystyle 1-v_{\text{max}}\frac{\dt}{\dx} &\displaystyle\le \frac{\partial F}{\partial \phi_{i}} \le& \displaystyle1 \vspace{4pt}
 \\
  \displaystyle0&\displaystyle\le \frac{\partial F}{\partial \phi_{i+1}} \le& \displaystyle v_{\text{max}}\frac{\dt}{\dx}.
\end{array}
\]
The CFL condition \eqref{CFLnl} ensures the second one to be nonnegative. The schemes \eqref{LimSchemer0}-\eqref{LimSchemenl} then fulfill the hypothesis of Theorem \ref{CrandallLionsThmnl}.
Therefore, it converges to the viscosity solution of \eqref{HJr0}-\eqref{HJnl} with the integrations in velocity done with $\lla \cdot\rra_{N_v}$, when the discretization parameters $\dt$ and $\dx$ tend to $0$.

\section{Numerical tests}
\label{SecTests}

In this section, we present some numerical tests for the scheme \eqref{Schemenl}. In all the simulations, we will use the discretizations \eqref{Discret_x}-\eqref{Discret_v}-\eqref{Discret_t}, with $x_{\text{max}}=1, \dv=1.25\cdot 10^{-2}$, and $v_{\max}=1$. The time and space steps $\dt$ and $\dx$ will be given in each case. %Spatially periodic boundary conditions are considered. 
Section \ref{SecScheme_fin} excepted, the equilibrium $M$ we consider is constant on $[-v_{\max},v_{\max}]$, such that $\lla M\rra_{N_v}=1$.
%\subsection{Accuracy and AP property of the scheme}

To compare the results given by \eqref{Schemenl} for large values of $\ep$, we will compute numerical solutions of \eqref{KinEq2bis} using an explicit scheme for \eqref{KinEq2bis}. Denoting $f^\ep(t_n,x_i,v_j)\sim f^n_{i,j}$ and $\rho^\ep(t_n,x_i)\sim\rho^n_i$, such  a scheme writes 
\begin{equation}
 \label{Explicitnl}
 \left\{
 \begin{array}{l}
  %\displaystyle \frac{\rho^{n+1}_i-\rho^n_i}{\dt}+\lla \left[ v\cdot \nabla_x f]^n_{i,j}\rra_{N_v} = 0 \\
  \displaystyle \frac{f^{n+1}_{i,j}-f^n_{i,j}}{\dt}+[v\partial_x f]^n_{i,j}=\frac{1}{\ep}\left(\rho^n_i M_j -f^n_{i,j} +r\rho^n_i\left(  M_j-f^ n_{i,j}\right) \right)\vspace{4pt} \\
  \displaystyle \rho^n_i=\lla f^n_{i,j}\rra_{N_v},
 \end{array}
 \right.
\end{equation}
where the gradient $[v\cdot\nabla_x f]^n_{i,j}$ is computed using an upwind scheme 
\[
 [v\partial_x f]^n_{i,j}=v_j^+\frac{f^n_{i,j}-f^n_{i-1,j}}{\dx}+v_j^-\frac{f^n_{i+1,j}-f^n_{i,j}}{\dx}.
\]
The values $\phi^n_i$ and $\eta^n_{i,j}$ are then deduced from the values of $f^n_{i,j}, \rho^n_i$
\begin{equation}
 \label{transfonl}
 \phi^n_i=-\ep\ln(\rho^n_i), \;\; \eta^n_{i,j}=-\ep\ln\left( \frac{f^n_{i,j}}{\rho^n_iM_j} \right).
\end{equation}
For large values of $\ep$, the scheme \eqref{Explicitnl} is accurate, since it is an order $1$ upwind scheme for an equation which contains no stiff terms.
%Moreover, when $\ep$ is not too small, the transformation \eqref{transfonl}, is also accurate. 
However, the computation of $\phi$ and $\eta$ for small values of $\ep$ with this method requires the explicit resolution of a stiff equation, which is costly from a computational point of view. 
Indeed, as a layer of width $\ep$ appears in the asymptotic regime, a condition  linking $\dx$ and $\ep$ must be satisfied to ensure the solution of \eqref{Explicitnl} to be accurate. Roughly, it writes $ \dx\le c\ep$.  In addition, the CFL condition 
\begin{equation}
 \label{CFLexp}
 %\dx\le c\ep, \;\;\; 
 \frac{\dt}{\dx}\le C,
\end{equation}
makes the numerical resolution of \eqref{Explicitnl} costly when $\ep$ is small.
% Indeed, to ensure the solution of \eqref{Explicitnl} to be accurate, the following CFL condition must be satisfied 
% \begin{equation}
%  \label{CFLexp}
%  \frac{\dt}{\ep \dx}\le C,
% \end{equation}
% which restrains strongly $\dt$ when $\ep$ is small. 
The solution of \eqref{Schemenl} will be compared to the numerical solution of the limit equation \eqref{HJr0}-\eqref{HJnl} computed with \eqref{LimSchemer0}-\eqref{LimSchemenl}. In this limit scheme, the numerical resolution of the nonlinear equation needed to find $H^{n+1}$ is performed with the Newton method proposed in \cite{LuoPayneEikonal}.
In this method,  
%the Newton method is slightly modified to ensure $\phi$ and the denominator to remain nonnegative along the iterations. When considering the micro-macro scheme \eqref{Schemenl} for small $\ep>0$, this constraint is automatically satisfied, thanks to the maximum principle.
the Newton method is used to solve \eqref{LimDefHnl}. It is however slightly modified to ensure $\phi^{n+1}_i$, defined in \eqref{LimSchemer0}-\eqref{LimSchemenl}, and the denominator of \eqref{LimDefHnl}, to remain nonnegative along the iterations.
When considering the micro-macro scheme \eqref{Schemenl} for small $\ep>0$, this constraint is automatically satisfied, thanks to the maximum principle.

\subsection{Consistency and AP character of the scheme in the linear case $r=0$}

In this section, we test the consistency and the AP property of the scheme \eqref{Schemenl} when $r=0$, and we consider spatially periodic boundary conditions. 
Firstly, we consider
%We consider in a first time 
the initial value 
\begin{equation}
\label{Phi_reg}
 \phi_{\tin}(x)=x^2, \;\; x\in[-1,1].
\end{equation}
To check  the consistency of the scheme \eqref{Schemenl}, we  test it for large values of $\ep$, for which the solution of \eqref{KinEqpsi} is far from the solution of the limit equation \eqref{HJr0}-\eqref{HJnl}. 
For $x\in[-1,1]$, we compute the solutions given by \eqref{Schemenl} for $\ep=1, 10^{-1}$, and $10^{-2}$,
%. For $\ep=1$, we consider the time step $\dt=5\cdot10^{-3}$.  Due to the condition \eqref{CFLexp}, for $10^{-1}$ we must refine in time and consider $\dt=2.5\cdot 10^{-1}$. 
and the parameters $\dt=2.5\cdot 10^{-3}$, and $\dx = 10^{-2}$. For $\ep=1$ and $\ep=10^{-1}$, the equation \eqref{KinEq2bis} can be solved numerically with \eqref{Explicitnl} and the same numerical parameters. The solutions given by \eqref{Schemenl}  are displayed in Fig.\;\ref{Phi_reg_ep_1} and Fig.\;\ref{Phi_reg_ep_10e-1}, they match with the solutions given by \eqref{Explicitnl} which are represented on the same figures.
%We compare the results we obtained with the solutions of \eqref{Explicitnl} with the same numerical parameters. The solutions are displayed in Fig. \ref{Phi_reg_ep_1} and  Fig. \ref{Phi_reg_ep_10e-1}.

\begin{figure}[!ht]
 \centering
\includegraphics[width=17cm]{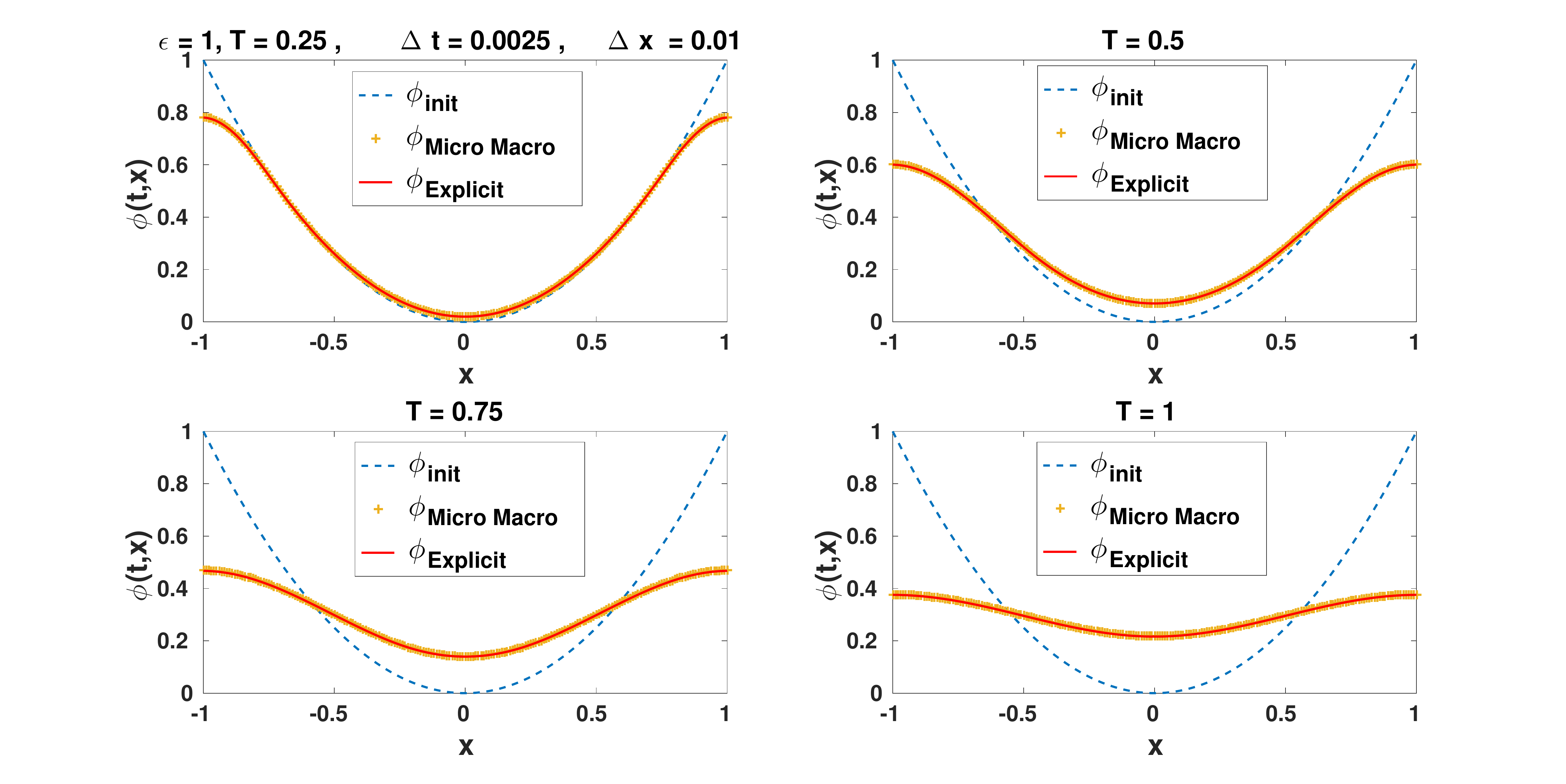}
\caption{The solutions of \eqref{Schemenl} and \eqref{Explicitnl} for $\ep=1$ at  times $T=0.25,\;0.5,\;0.75,\;1$, computed with $\dt=2.5\cdot 10^{-3}, \dx=10^{-2}$ and $\dv=1.25\cdot 10^{-2}$.}
\label{Phi_reg_ep_1}
\end{figure}

\begin{figure}[!ht]
 \centering
\includegraphics[width=17cm]{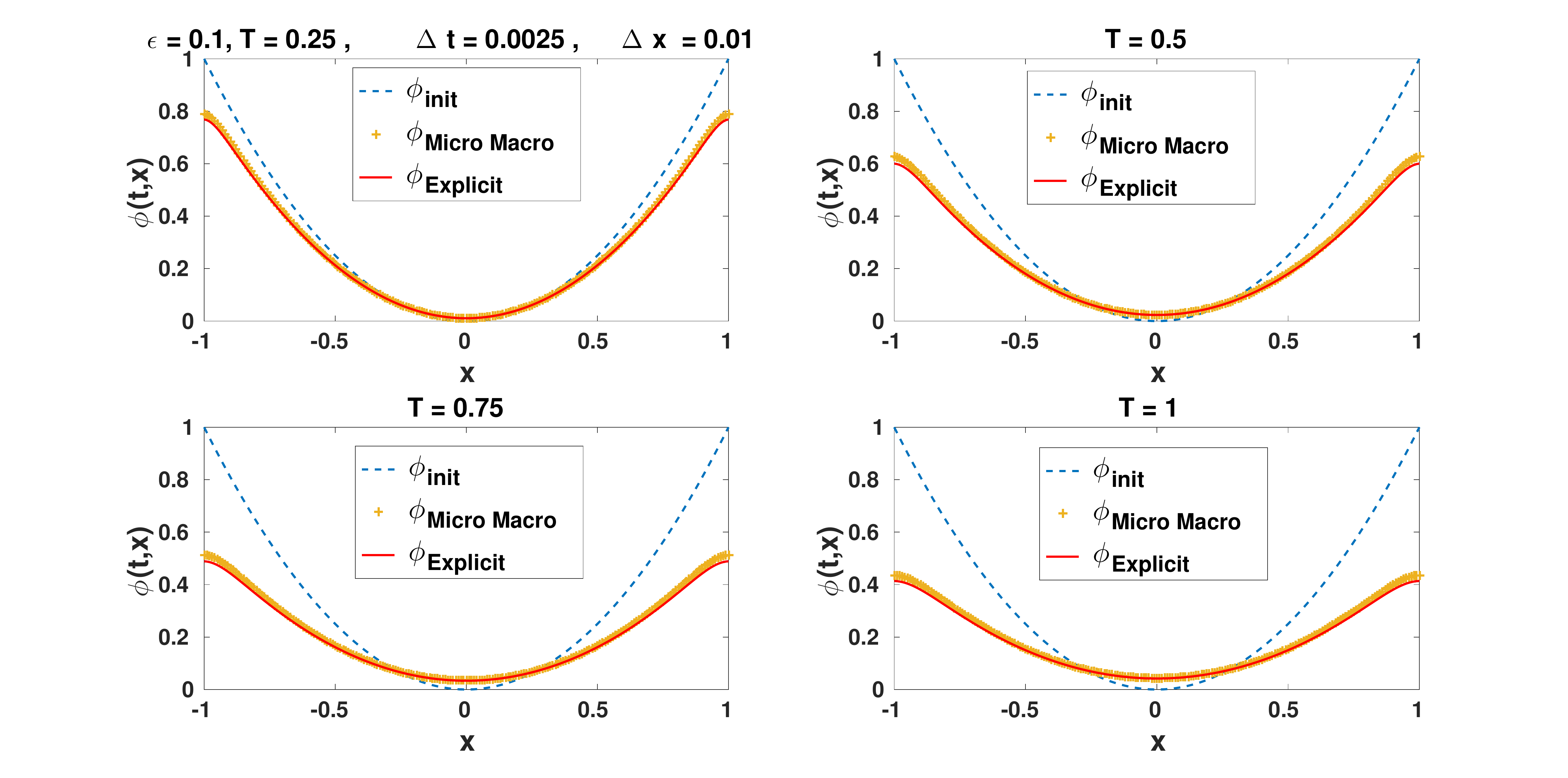}
\caption{The solutions of \eqref{Schemenl} and \eqref{Explicitnl} for $\ep=10^{-1}$ at  times $T=0.25,\;0.5,\;0.75,\;1$, computed with $\dt=2.5\cdot 10^{-3}, \dx=10^{-2}$ and $\dv=1.25\cdot 10^{-2}$.}
\label{Phi_reg_ep_10e-1}
\end{figure}

%\newpage
When $\ep$ becomes smaller, the computation of the solution of \eqref{KinEq2bis} with the scheme \eqref{Explicitnl} becomes  costly, because of the stiffness of the problem, and of the condition \eqref{CFLexp}.  
The solution of \eqref{Schemenl} for $\ep=10^{-2}$ is then compared to the solution of the limit scheme \eqref{HJr0} in Fig. \ref{Phi_reg_ep_10e-2}. Once again, both are computed with $\dt=2.5\cdot 10^{-3}$ and $\dx=10^{-2}$. We observe that for such a small value of $\ep$, the solution of the kinetic problem \eqref{KinEqpsi}, computed with \eqref{Schemenl} is close to the solution of the asymptotic problem \eqref{HJr0}. These tests, done independently for $\ep=1, 10^{-1}$ and for $\ep=10^{-2}$, show that the scheme \eqref{Schemenl} is accurate for the values of $\ep$ for wich the problem is not stiff, and that it catches the viscosity solution of the limit problem \eqref{HJr0} when $\ep$ is smaller. 
\begin{figure}[!ht]
 \centering
\includegraphics[width=17cm]{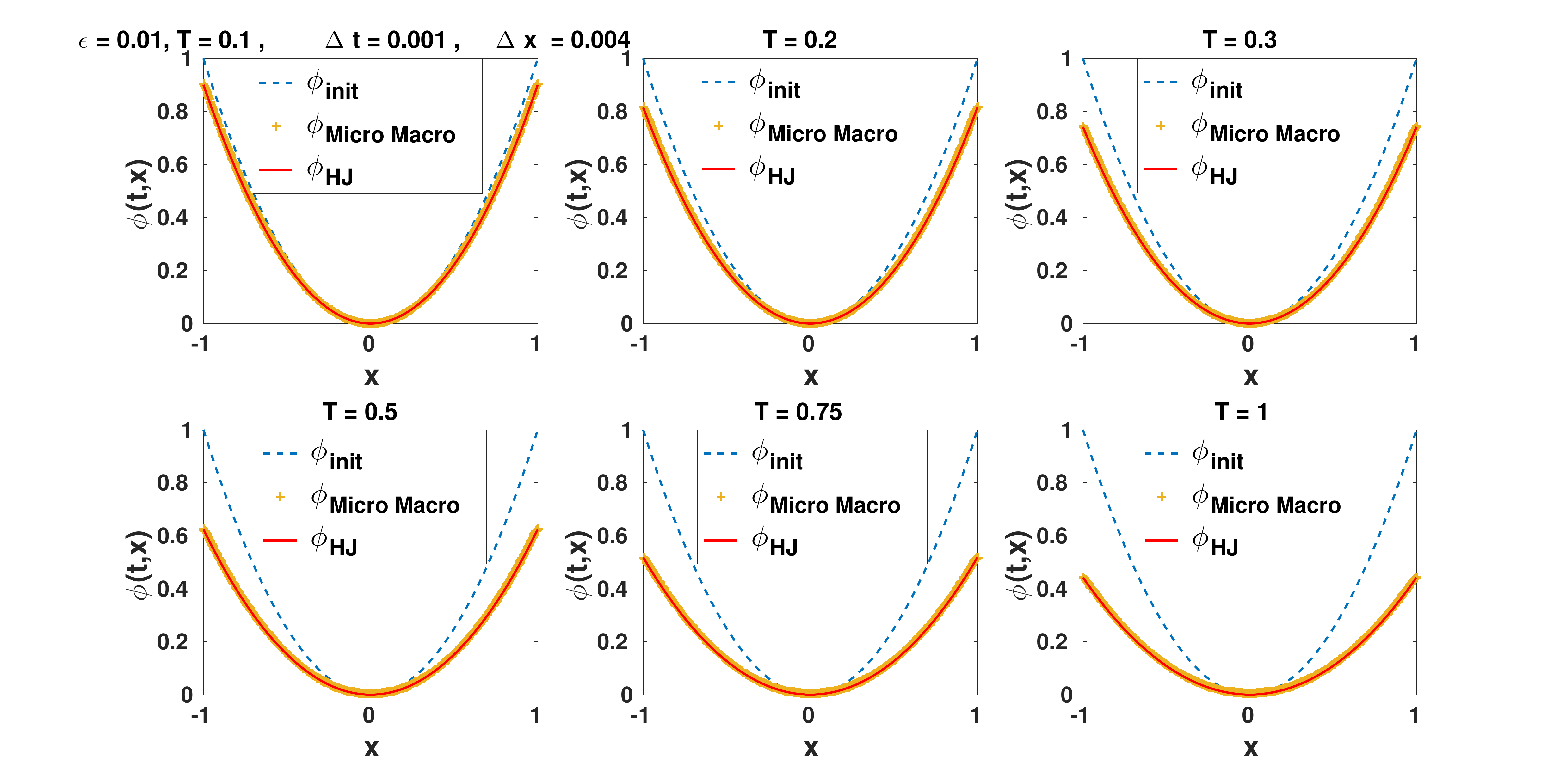}
\caption{The solutions of \eqref{Schemenl} and \eqref{LimSchemer0} for $\ep=10^{-2}$ at  times $T=0.1,\;0.2,\;0.3,\;0.5,\;0.75,\;1$, computed with $\dt=2.5\cdot 10^{-3}, \dx=10^{-2}$ and $\dv=1.25\cdot 10^{-2}$.}
\label{Phi_reg_ep_10e-2}
\end{figure}

In both cases $r=0$ and $r>0$, the solution of \eqref{KinEqpsi} converges, when $\ep$ goes to $0$, to the viscosity solution of an Hamilton-Jacobi equation \eqref{HJr0}-\eqref{HJnl}. Depending on the properties of the initial data, the solutions of these equations may not be smooth. Indeed, if there are two distinct minimum points in the initial data, 
the solutions of \eqref{HJr0}-\eqref{HJnl} lack the $\mathcal{C}^1$ regularity after some time.
%a shock appears in the solutions of \eqref{HJr0}-\eqref{HJnl}.
Theorem \ref{CrandallLionsThmnl} ensures that the limit schemes \eqref{LimSchemer0}-\eqref{LimSchemenl} catch the viscosity solution of \eqref{HJr0}-\eqref{HJnl}. As a consequence, they are robust when 
%a shock 
such a lack of regularity
appears during the computations. 
However, to ensure that a scheme for \eqref{KinEqpsi} enjoys the AP property, it is necessary to make sure that it is robust in the case of an initial data leading to 
%a shock, 
a solution which does not enjoy the $\mathcal{C}^1$ regularity,
for all the values of $\ep$.  
%This property is ensured by Theorem \ref{CrandallLionsThmnl} for the limit schemes \eqref{LimSchemer0}-\eqref{LimSchemenl}.
In order to test the robustness of \eqref{Schemenl} in this case, we now consider an initial data with two minimum points. Once again, for large values of $\ep$, the solution of \eqref{Schemenl} is compared to the solution of \eqref{Explicitnl}, to check 
%that they match. 
the agreement between them.
The values  $\ep=1$ and $\ep=10^{-1}$ are considered in Fig. \ref{Phi_ep_1} and Fig. \ref{Phi_ep_10e-1}, with the parameters $\dt=10^{-3}$ and $\dx=4\cdot 10^{-3}$. 
\begin{figure}[!ht]
 \centering
\includegraphics[width=17cm]{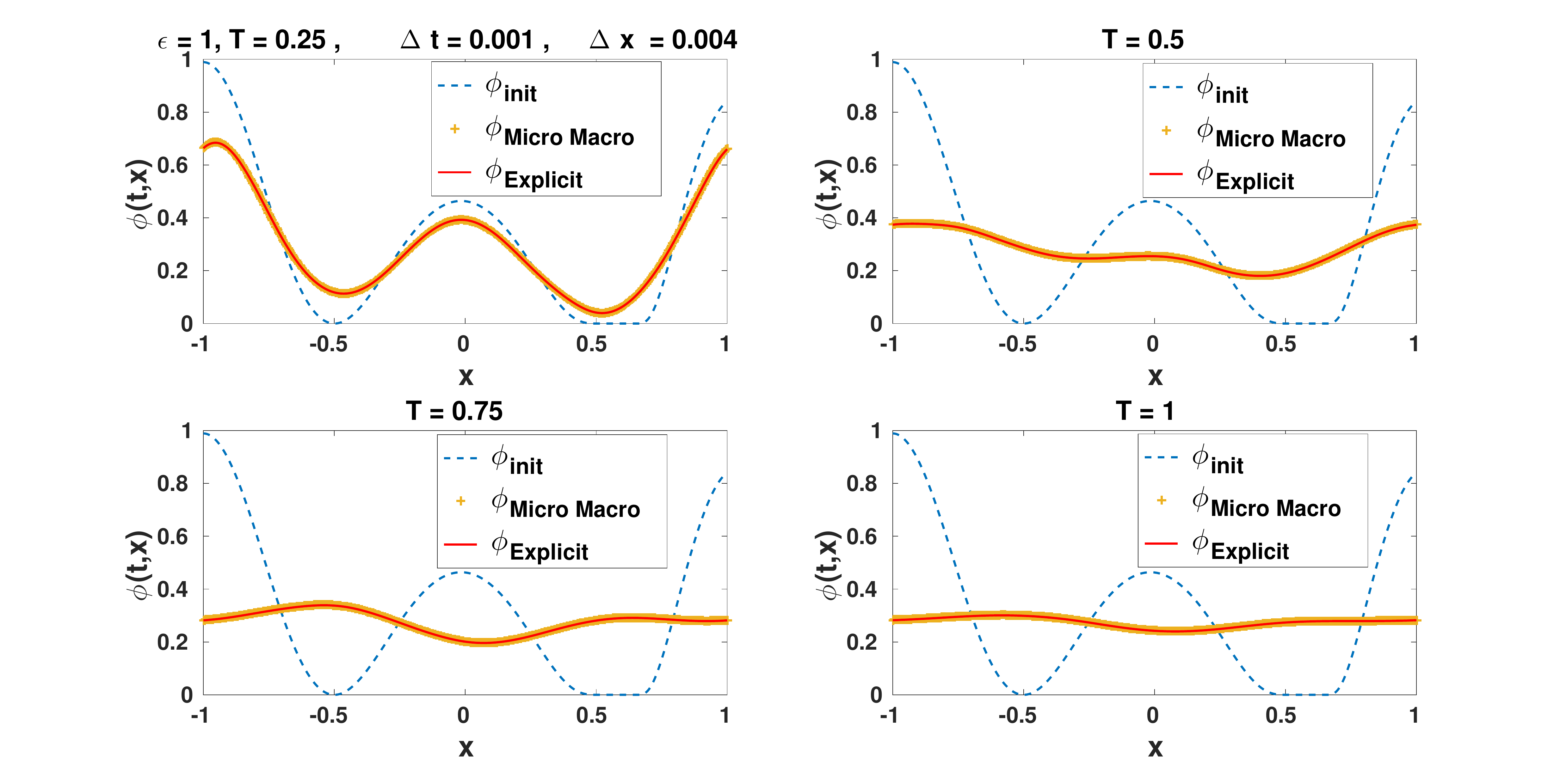}
\caption{The solutions of \eqref{Schemenl} and \eqref{Explicitnl} for $\ep=1$ at  times $T=0.25,\;0.5,\;0.75,\;1$, computed with $\dt=10^{-3}, \dx=4\cdot 10^{-3}$ and $\dv=1.25\cdot 10^{-2}$.}
\label{Phi_ep_1}
\end{figure}

\begin{figure}[!ht]
 \centering
\includegraphics[width=17cm]{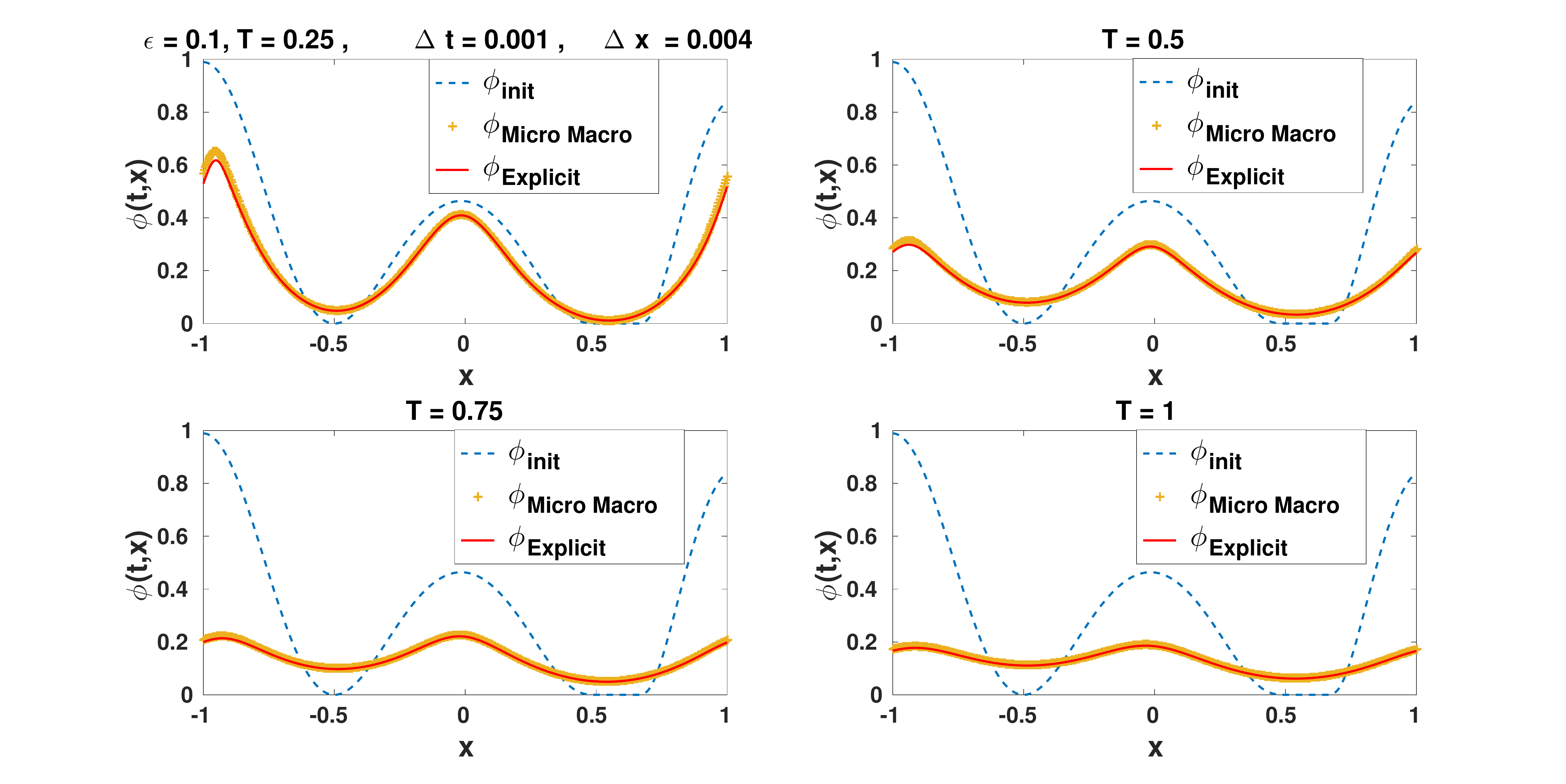}
\caption{The solutions of \eqref{Schemenl}-\eqref{DefetaHnl} and \eqref{Explicitnl} for $\ep=10^{-1}$ at  times $T=0.25,\;0.5,\;0.75,\;1$, computed with $\dt=10^{-3}, \dx=4\cdot 10^{-3}$ and $\dv=1.25\cdot 10^{-2}$.}
\label{Phi_ep_10e-1}
\end{figure}

%\newpage
When $\ep$ is smaller, the solution of the scheme \eqref{Schemenl} is compared to the solution of \eqref{HJr0}-\eqref{HJnl}. 
Still because of the stiffness of the problem \eqref{KinEq2bis} for small $\ep$, when $\ep$ is smaller it is costly to compare the solution of the micro-macro scheme \eqref{Schemenl} to the solution of the explicit scheme \eqref{Explicitnl}. In addition, it is not robust when the solution 
%presents some shocks, 
is not regular enough,
as in the test case we are considering. Instead, we compare the results given by the micro-macro scheme \eqref{Schemenl} to the solutions given by the limit scheme \eqref{LimSchemer0}, which catches the viscosity solution of the limit problem \eqref{HJr0}. The computations are done with the parameters $\dt=10^{-3}$, and $\dx=4\cdot 10^{-3}$, and 
the results are displayed in Fig. \ref{Phi_ep_10e-2} for $\ep=10^{-2}$ and Fig. \ref{Phi_ep_10e-3} for $\ep=10^{-3}$.
%\textcolor{red}{Sur celles là, on voit qu'on décroche un peu de $\phi=0$ sur le minimum de gauche}
\begin{figure}[!ht]
 \centering
\includegraphics[width=17cm]{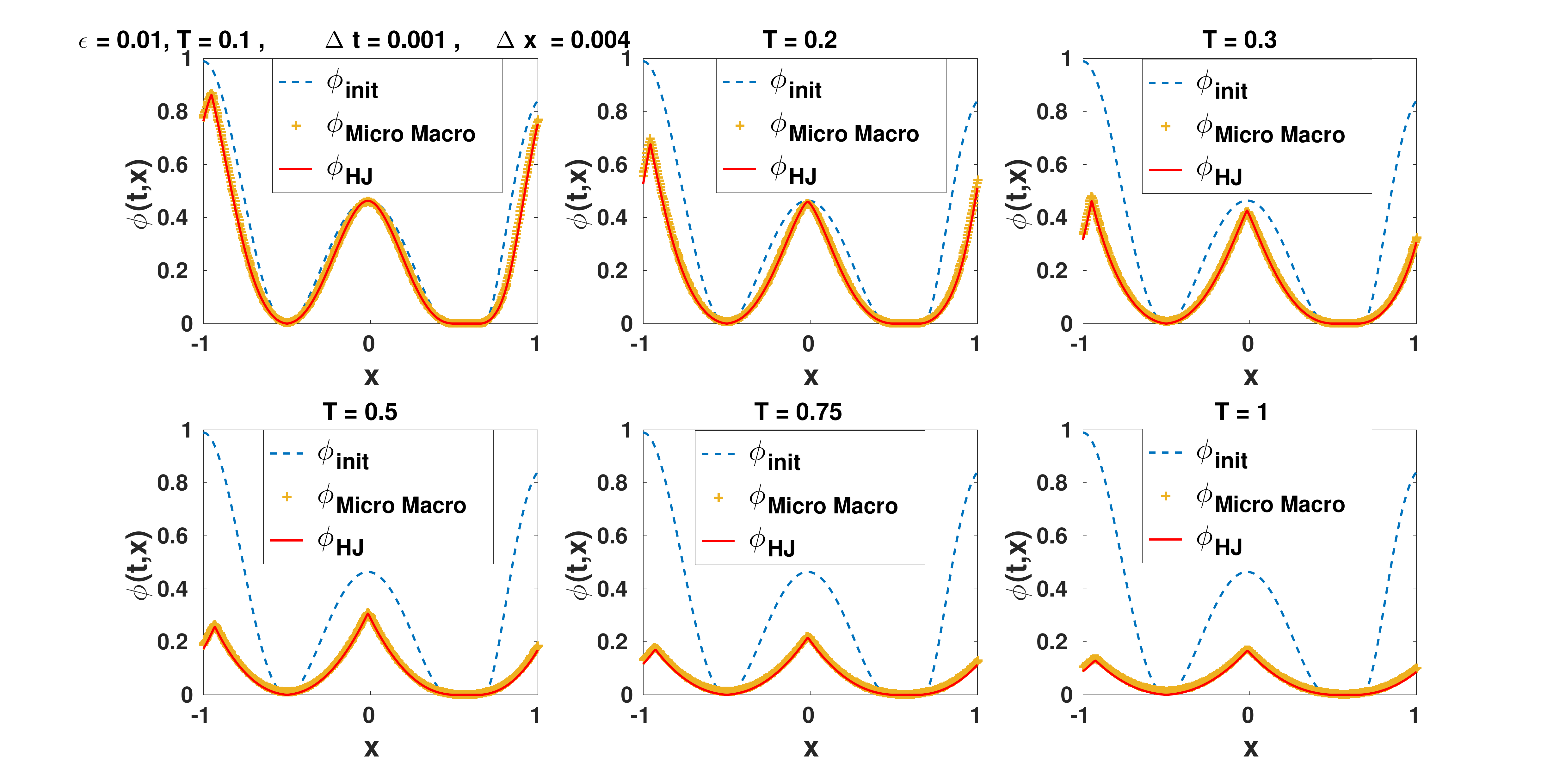}
\caption{The solutions of \eqref{Schemenl} and \eqref{LimSchemer0} for $\ep=10^{-2}$ at  times $T=0.1,\;0.2,\;0.3,\;0.5,\;0.75,\;1$, computed with $\dt=10^{-3}, \dx=4\cdot 10^{-3}$ and $\dv=1.25\cdot 10^{-2}$.}
\label{Phi_ep_10e-2}
\end{figure}

\begin{figure}[!ht]
 \centering
\includegraphics[width=17cm]{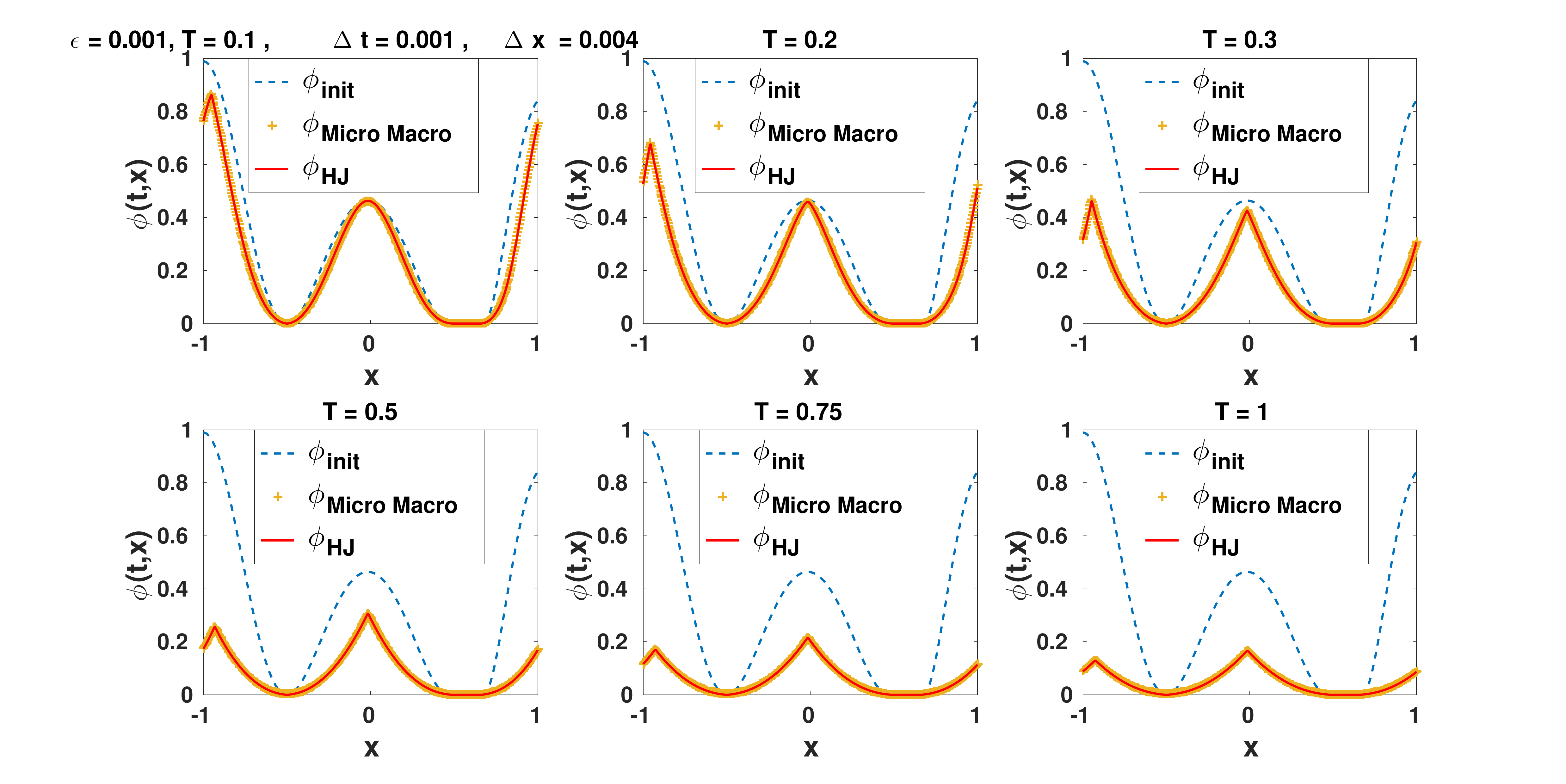}
\caption{The solutions of \eqref{Schemenl} and \eqref{LimSchemer0} for $\ep=10^{-3}$ at  times $T=0.1,\;0.2,\;0.3,\;0.5,\;0.75,\;1$, computed with $\dt=10^{-3}, \dx=4\cdot 10^{-3}$ and $\dv=1.25\cdot 10^{-2}$.}
\label{Phi_ep_10e-3}
\end{figure}

\newpage
\subsection{Consistency and AP character in the non-linear case $r>0$}

In this section, we study the consistency and the AP property of the micro-macro scheme \eqref{Schemenl} in the non-linear case $r>0$. As in the linear case, the accuracy of the micro-macro scheme \eqref{Schemenl} is tested for different values of $\ep$. The results are displayed in Fig. \ref{test_consistanceAP}, where the phase $\phi$ given by the scheme is compared to the solution of an explicit scheme \eqref{Explicitnl} for \eqref{KinEq2} when $\ep=1$ or $\ep=10^{-1}$, and to the solution of the limit scheme \eqref{LimSchemenl} for $\ep=10^{-2}$ or $\ep=10^{-3}$. All the results are presented at time $T=0.5$, and have been computed with $\dx=10^{-2}$, and $\dt=2.5\cdot 10^{-3}$, and spatial periodic boundary conditions. As in the linear case, the comparison with the explicit scheme test shows that the scheme is accurate for large values of $\ep$, and the comparison with the limit scheme highlights its AP character.
\begin{figure}[!ht]
 \centering
\includegraphics[width=17cm]{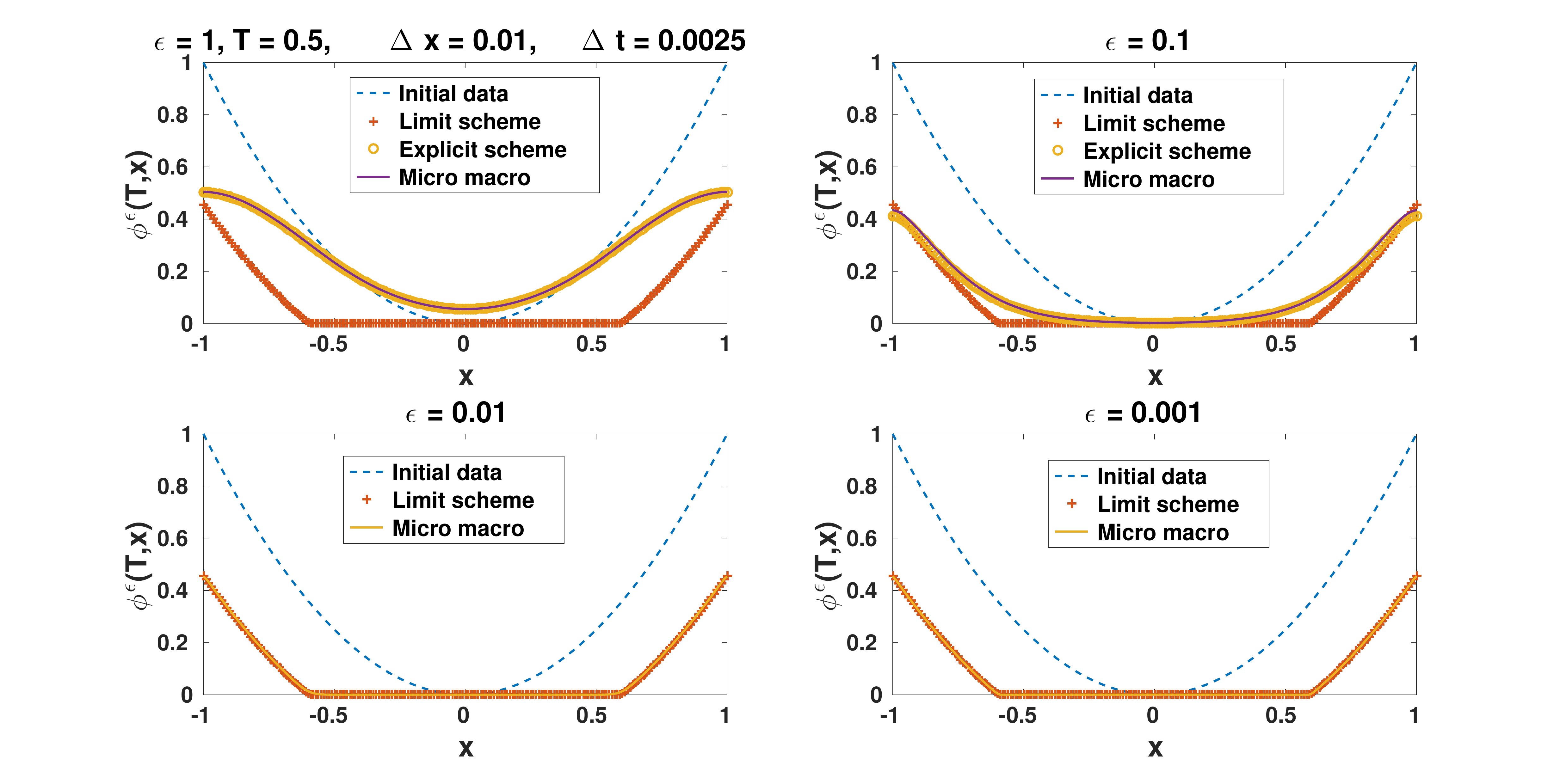}
\caption{The solutions of \eqref{Schemenl}, \eqref{Explicitnl}, and \eqref{LimSchemer0} for some values of $\ep$ at  time $T=0.5$, computed with $\dt=2.5\cdot10^{-3}, \dx=\cdot 10^{-2}$ and $\dv=1.25\cdot 10^{-2}$.}
\label{test_consistanceAP}
\end{figure}

When $r>0$, 
%it is expected that a front of bacteria propagate at a fixed speed $c^*$, which depends on the model. 
it is expected that the nullset of $\phi$ spreads at a fixed speed  $c^*$, which depends on the model. 
Indeed, the following result has been proved in \cite{Bouin},
\begin{prop}
 Assume that 
 \[
  \phi_0(x)=\left\{  
  \begin{array}{l l}
   0 & x=0 \\ +\infty & \text{\;else\;}
  \end{array},
  \right.
 \]
and define
\[
 c^*=\inf\limits_{p>0}\left( \frac{H(p)+r}{p}\right).
\]
Then the nullset of the solution $\phi$ of \eqref{HJnl} propagates at speed $c^*$:
\[
 \forall t>0,  \;\; \left\{\phi(t,\cdot )=0\right\}=B(0,c^* t).
\]
\end{prop} 

The tests we propose in what follows are designed to highlight the propagation of fronts that is expected, thanks to the positivity of $r$. The speed of the propagation of the fronts is also tested at the numerical level. 
%To match with the experiments that can be done to see bacteria \emph{E. coli} propagate along an unidimensional channel (see \cite{CalvezEColi}),
The initial condition we consider is a density $\rho$ such that $\rho=1$ at the left of the spatial domain and $0$ elsewhere, with Neumann spatial boundary conditions. To ensure the test is done in the asymptotic regime, we consider $\ep=10^{-4}$ with the parameters $\dt= 3.125\cdot 10^{-4} $, and $\dx = 1.25\cdot 10^{-3} $. Such a refined grid is not necessary to observe the propagation of the front, but it provides a better accuracy on the computation of the numerical propagation speed. The density given by the micro-macro scheme \eqref{Schemenl} at different times is displayed in Fig. \ref{nonlinear_propagation}. The speed of the propagation of the front is computed in the left part of Fig. \ref{nonlinear_propagationspeed}. To determine numerically its theoretical value, the minimum of 
\begin{equation}
\label{cp}
 c(p)=\frac{H(p)+r}{p},
\end{equation}
for $p>0$ is computed. It is presented in the right side of Fig. \ref{nonlinear_propagationspeed}. The precision of the result depends on $\dx$, as highlighted in Fig. \ref{PrecisionPropagationSpeed}. In this figure, the relative error $|c_{\dx}^*-c^*|/c^*$, where $c^*_{\dx}$ denotes the numerical propagation speed, is presented as a function of $\dx$. In a logarithmic scale, a line is obtained, with slope $2$.
%Since a propagation phenomenon is expected for the positive values of $r$, the tests we propose are designed to highli
\begin{figure}[!ht]
 \centering
\includegraphics[width=17cm]{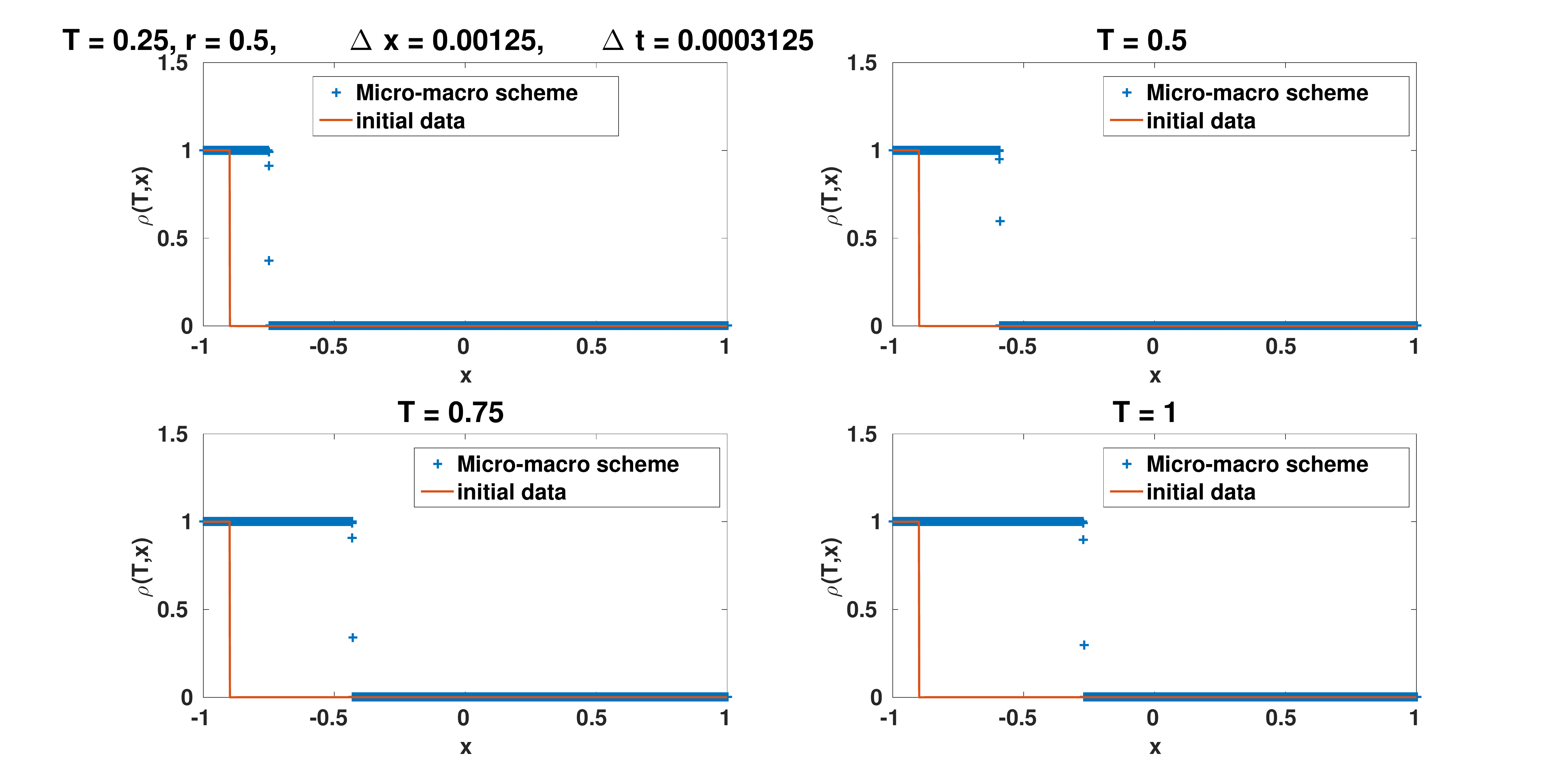}
\caption{The solution of the micro-macro scheme \eqref{Schemenl} for  $\ep=10^{-4}$ at  times $T=0.25,\; 0.5, \;0.75$ and $1$, computed with $\dt= 3.125\cdot 10^{-4}, \dx=1.25\cdot 10^{-3}$ and $\dv=1.25\cdot 10^{-2}$.}
\label{nonlinear_propagation}
\end{figure}

\begin{figure}[!htbp]
\begin{center}
\begin{tabular}{@{}c@{}c@{}}
\includegraphics[width=8cm]{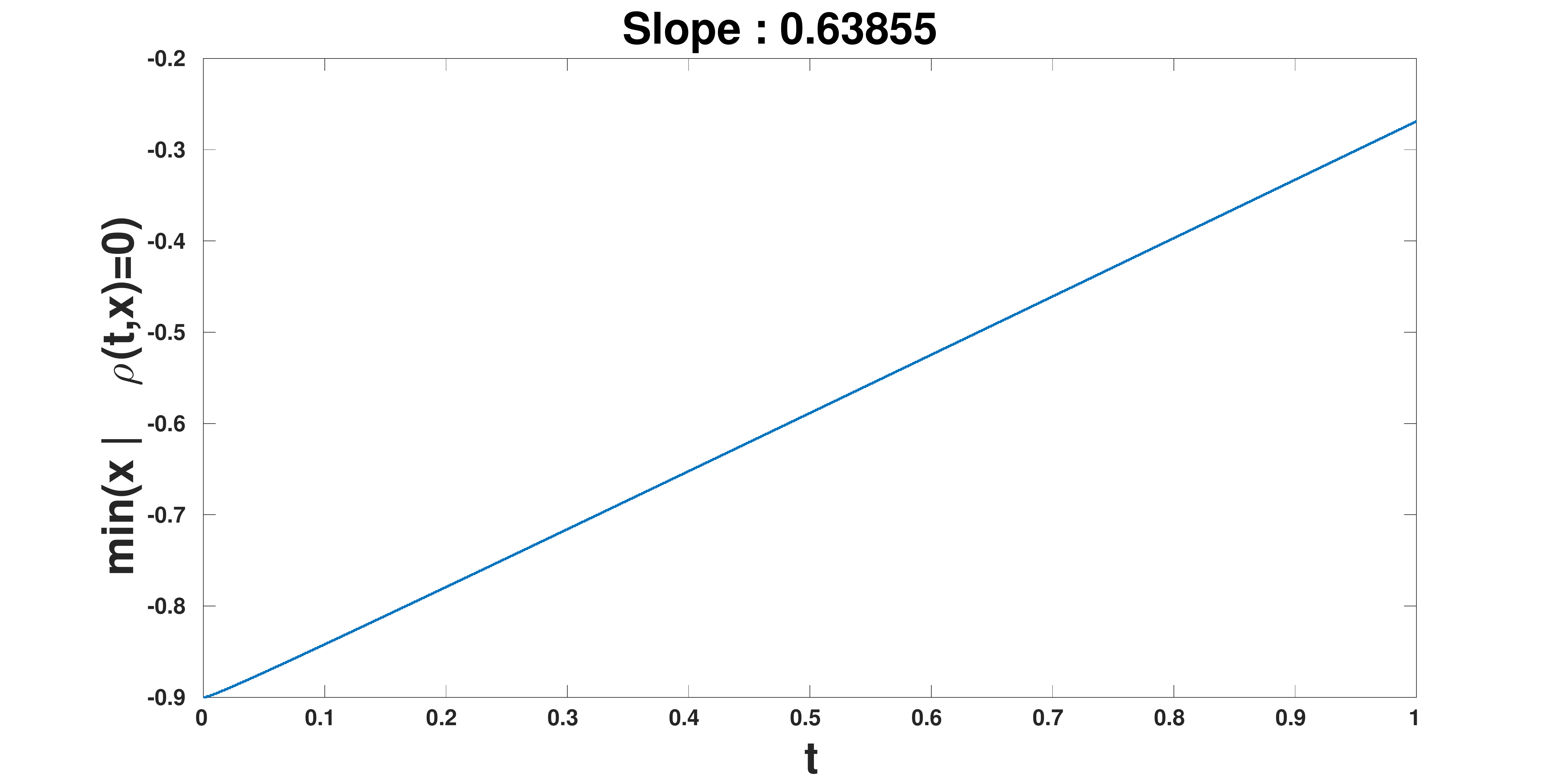}    &
\includegraphics[width=8cm]{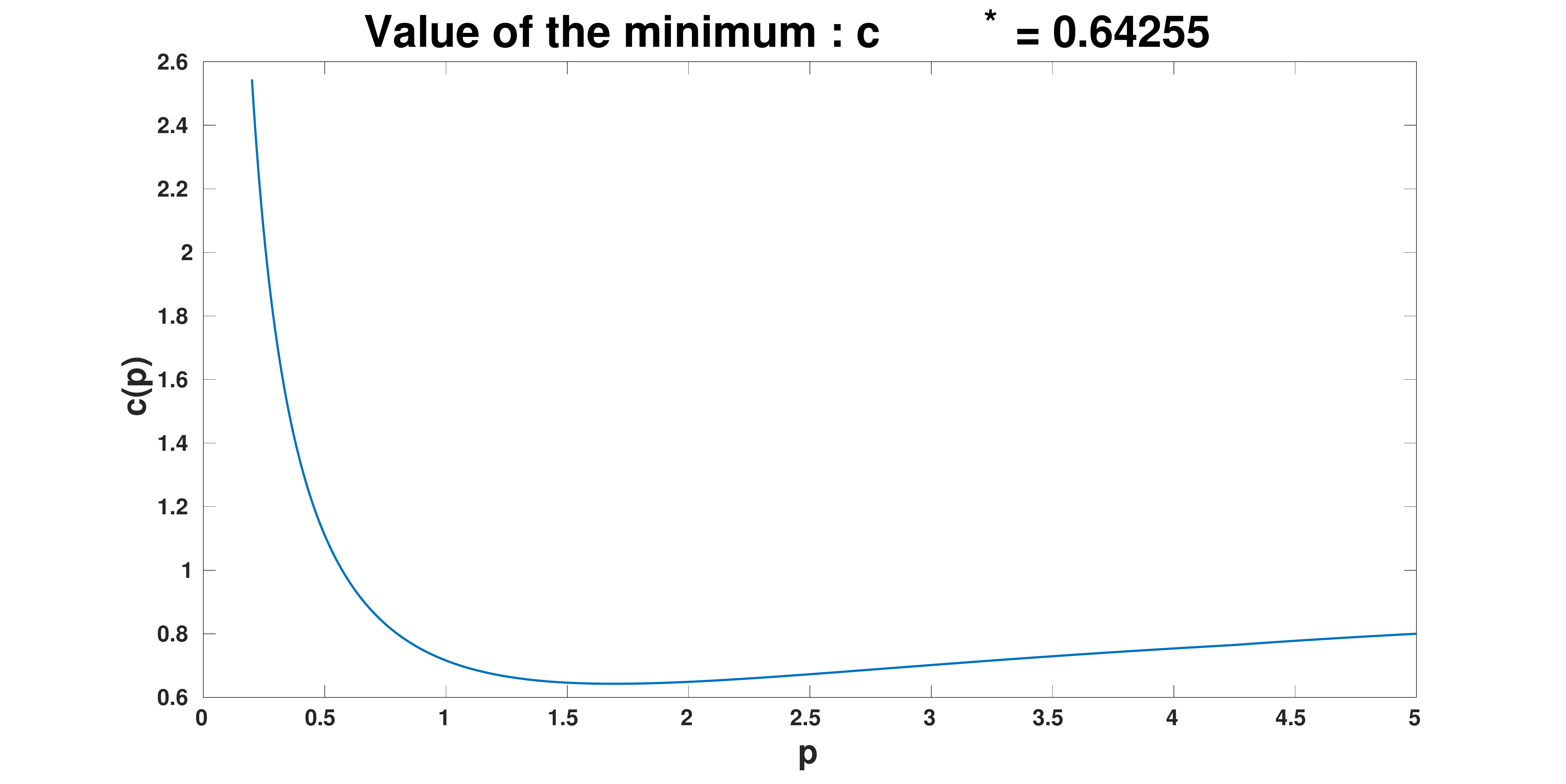}
\end{tabular}
\caption{Left: The place of the front $\min\limits_x\{\rho(t,x)=0\}$ as a function of $t$, computed with $\dt= 3.125\cdot 10^{-4}, \dx=1.25\cdot 10^{-3}$ and $\dv=1.25\cdot 10^{-2}$. The slope of the line gives the numerical propagation speed of the front. Right : The quantity $c(p)$ \eqref{cp} as a function of $p$. The theoretical front propagation speed is the minimal value of $c(p)$.}
\label{nonlinear_propagationspeed}
\end{center}
\end{figure}

\begin{figure}[!htbp]
 \begin{center}
  \includegraphics[width=13cm]{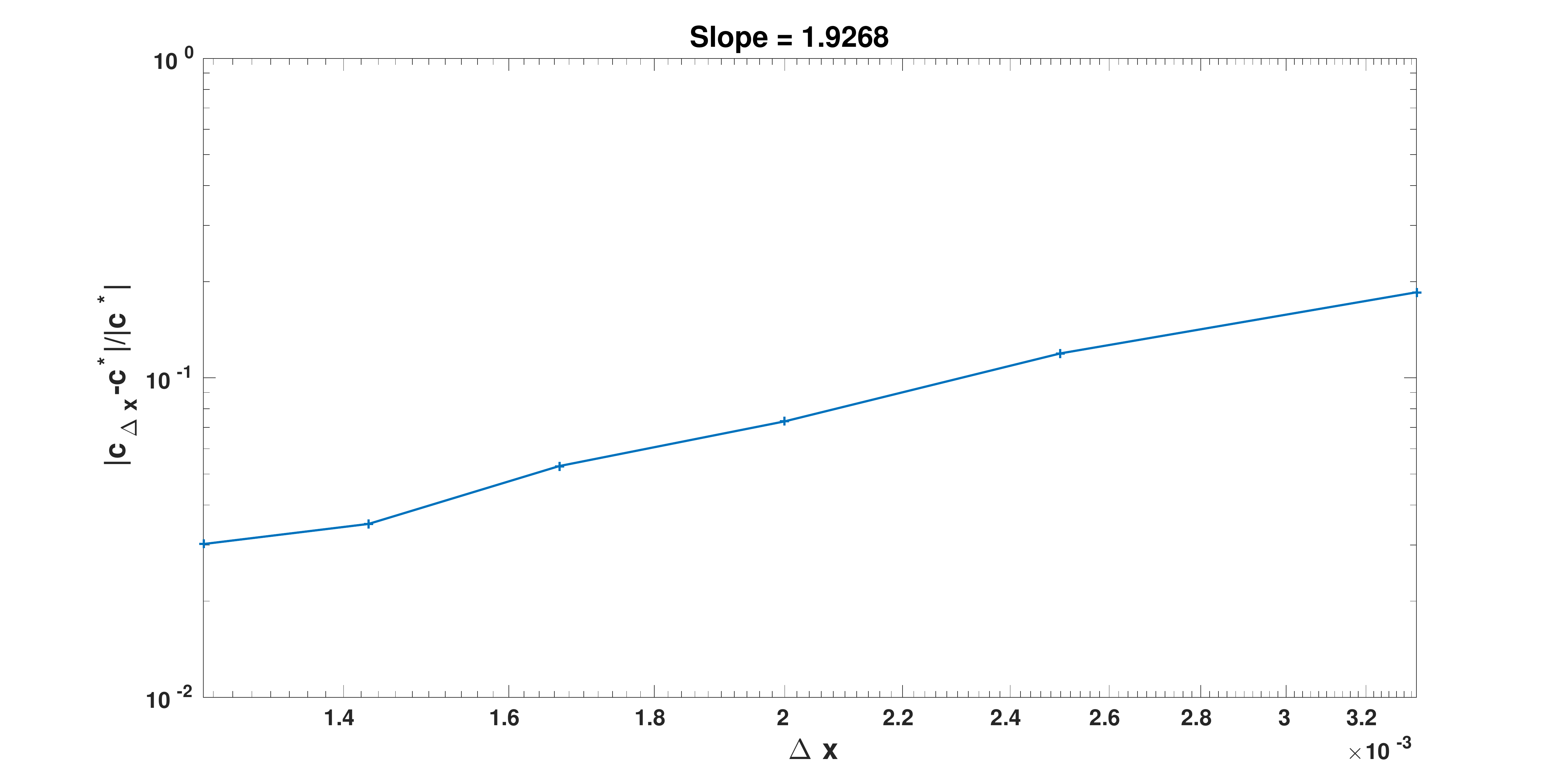}
  \caption{The relative error $|c_\dx^*-c^*|/c^*$ as a function of $\dx$ (log scale).}
  \label{PrecisionPropagationSpeed}
 \end{center}
\end{figure}

\newpage
\subsection{Order and uniform accuracy}
\label{Sec_Scheme_UA}

In this section, we study the order of the scheme when $\ep$ is fixed, and we investigate its uniform accuracy.
To study the order of the scheme, 
we choose the initial data \eqref{Phi_reg}, and
we define a reference solution as the solution given by \eqref{Schemenl}-\eqref{DefetaHnl} for $\dx=2\cdot10^{-3}$, $\dt=5\cdot 10^{-4}$, and $\dv=1.25\cdot 10^{-2}$. 
%We fix $\dv=1.25\cdot 10^{-2}$ and $\dt$ and 
Keeping $\dt$ and $\dv$ fixed, 
we compute the solutions of \eqref{Schemenl} for different values of $N_x$. 
The error
%We then compute the error 
\[
 E(\ep,\dx)=\frac{\|\phi^\ep_{ref}-\phi^\ep_{\dx}\|_\infty}{\|\phi^\ep_{ref}\|_\infty},
\]
is then computed for $\ep=1$.
It is displayed in Fig. \ref{Order} in logarithmic scale. As the slope of the line is slightly greater than $1$, the scheme is of order $1$ in space, as expected for an upwind scheme. 

\begin{figure}[!ht]
 \centering
\includegraphics[width=17cm]{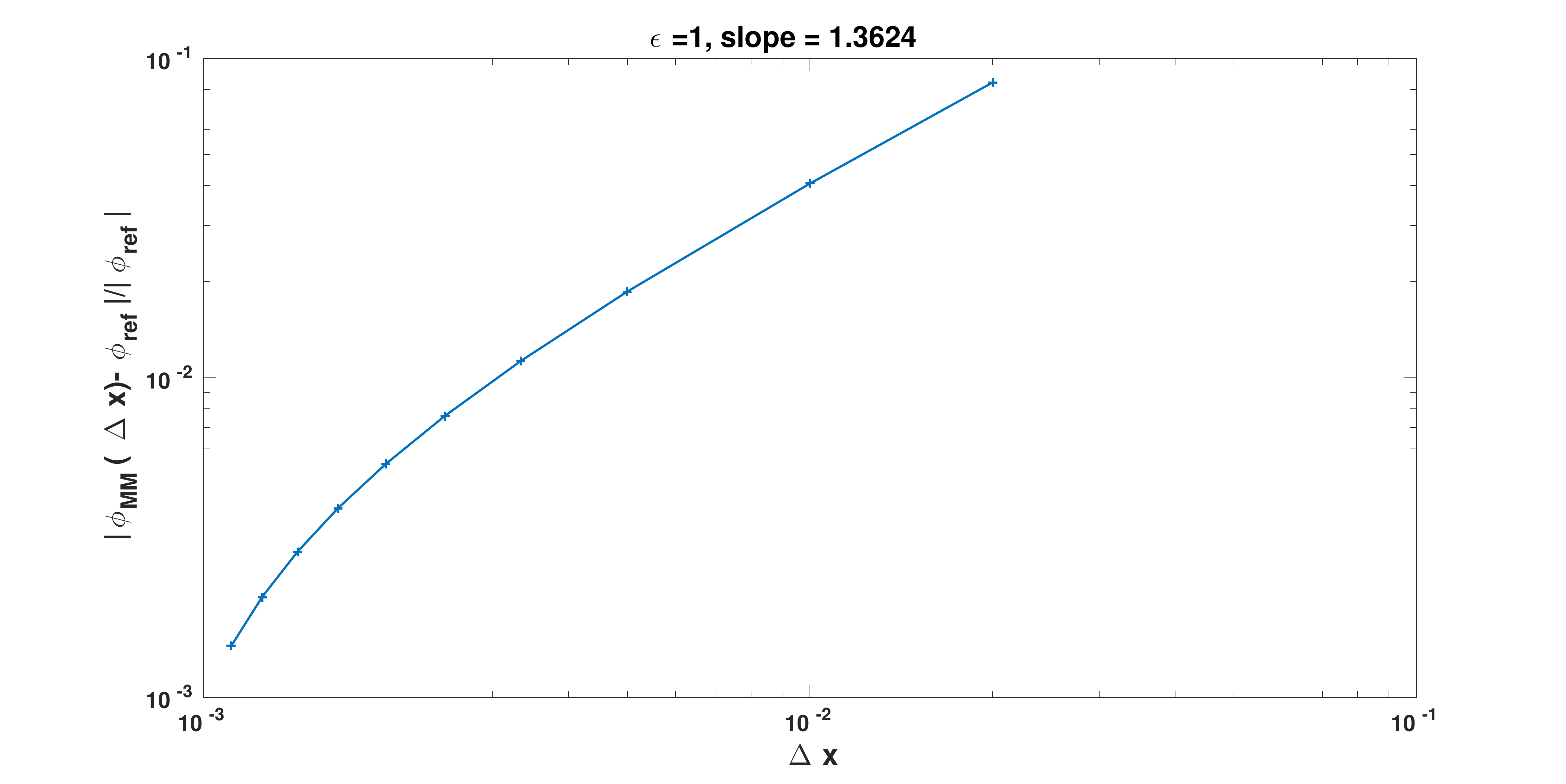}
\caption{The expression $E(1,\dx)$ in function of $\dx$ (log scale).}
\label{Order}
\end{figure}

The scheme enjoys the UA property, with uniform order $1$  if 
\[
 \sup\limits_{\ep\in(0,1]} E(\ep,\dx)\le C \dx,
\]
with $C$ independent of $\ep$. This property is
is highlighed in Fig. \ref{UA} where the error $E(\ep,\dx)$ is displayed in function of $\ep$ for different values of $\dx$. As the lines are stratified, the scheme \eqref{Schemenl}-\eqref{DefetaHnl} is uniformly accurate in $\ep$.

\begin{figure}[!ht]
 \centering
\includegraphics[width=17cm]{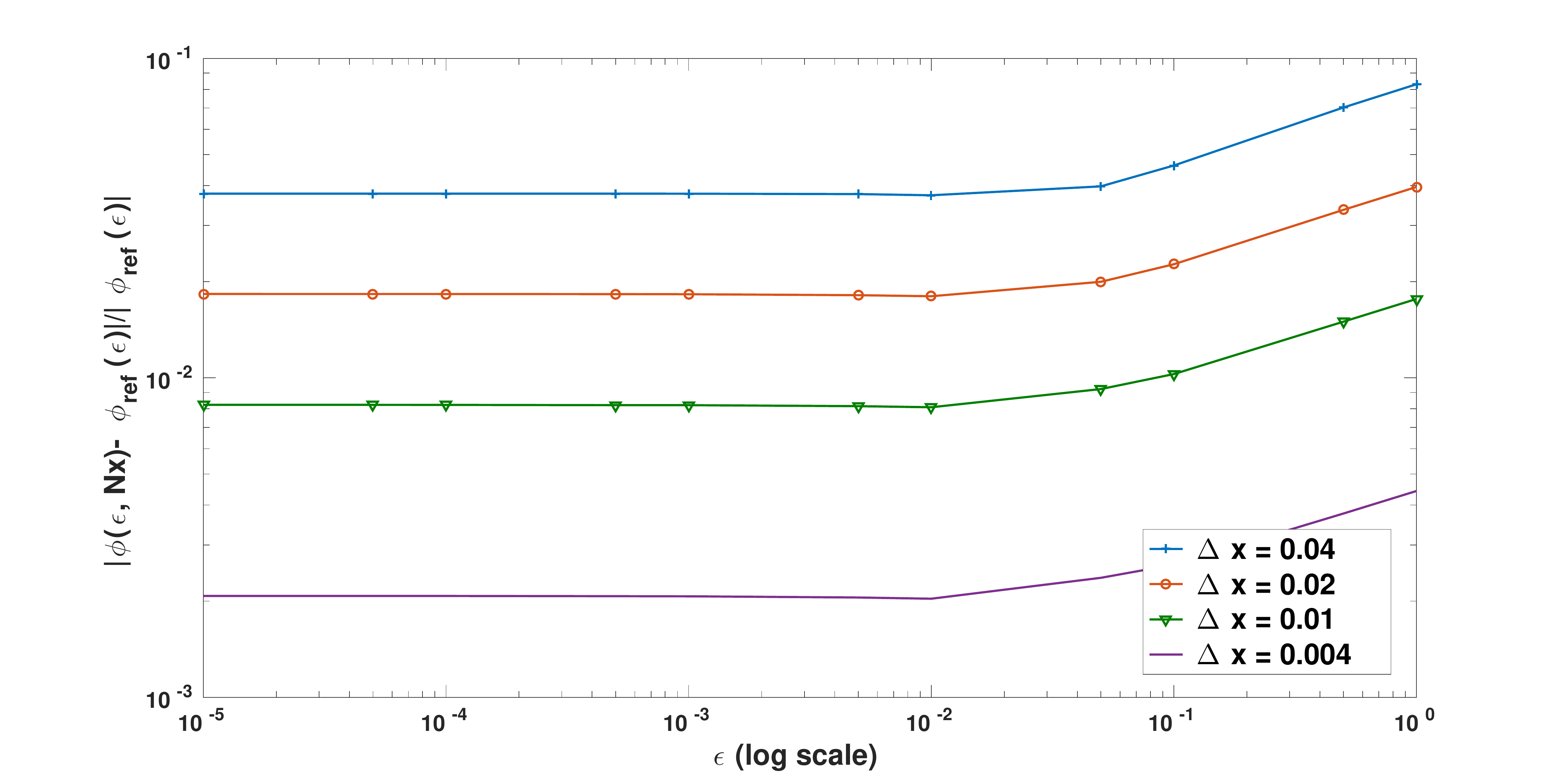}
\caption{The expression $E(\ep,\dx)$ in function of $\ep$ for different values of $\dx$ (log scale).}
\label{UA}
\end{figure}

\subsection{Case of a singular equilibrium}
\label{SecScheme_fin}

In this section, we consider the linear case $r=0$, and an equilibrium $M$ vanishing at some point of the velocity domain. It has been proved in 
\cite{Caillerie}
 that, in this case, the asymptotic model is still a Hamilton-Jacobi equation, but that the Hamiltonian $H$ is not given by the implicit formula \eqref{Hnl} anymore.
%Indeed, $H$ is the solution of the equation
% \[
%  \lla \frac{M(v)}{1+H(p)-v\cdot p}\rra =1,
% \]
% but if $M$ vanishes, 
Indeed, if we suppose that the velocity space $V$, which is the support of $M$,
%is a bounded closed convex subset of $\R^d$ 
writes $V=[-v_{\max},v_{\max}]$
(such an hypothesis simplifies the notations but is not necessary), and if we denote
\begin{equation}
 \label{mup}
 \mu(p)=%\max\limits_{v\in V} v\cdot p,
 v_{\max} |p|,
\end{equation}
and 
% \begin{equation}
%  \label{Argmup}
%  \mathrm{Arg} \;\mu(p)=\left\{v\in V,\; v p=\mu(p)\right\}, 
% \end{equation}
% and 
\begin{equation}
 \label{SingM}
 \mathrm{Sing}(M)=\left\{ p\in\R ,\;\lla \frac{M(v)}{\mu(p)-v p}\rra\le1\right\},
\end{equation}
the following result holds
\begin{thm}
 \label{thmCaillerie}
Suppose that $\phi^\ep(0,x,v)=\phi_\tin(x)$, then $\phi^\ep$ converges locally uniformly on $\R_+\times\R\times V$ towards $\phi$, where $\phi$ does not depend on $v$. Moreover, $\phi$ is the viscosity solution of the following Hamilton-Jacobi equation:
\begin{equation}
\label{FinLimEq}
 \partial_t \phi(t,x)+H\left(\partial_x\phi(t,x)\right)=0, \;\; (t,x)\in\R_+\times\R,
\end{equation}
where the Hamiltonian $H$ is defined as follows: if $p\in \mathrm{Sing}(M)$, then $H(p)=\mu(p)-1$, else $H(p)$ is uniquely determined by the following implicit formula 
\begin{equation}
 \label{DefH}
 \lla \frac{M}{1+H(p)-v p}\rra =1.
\end{equation}
\end{thm}
This theorem, which has been proved in \cite{Caillerie} can be explained formally. Indeed, with the previous notations
\[
 \rho^\ep=\e^{-\phi^\ep/\ep}, \;\; \frac{f^\ep}{\rho^\ep M}=\e^{-\eta^\ep/\ep},
\]
the equation \eqref{KinEq2} reads, for $r=0$
\[
 \partial_t \phi^\ep +v\partial_x\phi^\ep =1-\e^{\eta^\ep/\ep},
\]
which can be reformulated as follows
\[
% (H^\ep-v p^\ep +1)M\e^{-\eta^\ep/\ep}=\lla M\e^{-\eta^\ep/\ep}\rra,
(H^\ep-v p^\ep +1)\e^{-\eta^\ep/\ep}=1,
\]
where $H^\ep=-\partial_t \phi^\ep$ and $p^\ep=\partial_x\phi^\ep$. Since $\lla M\e^{-\eta^\ep/\ep}\rra=1$, when $\ep\to 0$, $Q^\ep=\e^{-\eta^\ep/\ep}$ converges formally to the solution $Q$ of the spectral problem
\begin{equation}
\label{specprob}
 (H-v\cdot p +1)Q = \lla MQ\rra,
\end{equation}
 where $H=-\partial_t \phi$, $p=\partial_x \phi$. It implies that $1+H(p)-v p\ge 0$, and then that $H(p)\ge \mu(p)-1$. We distinguish between two cases. The first one is the same we treated in this paper, namely if $p\in \mathrm{Sing}(M)^c$, then 
\[
 \lla \frac{M}{\mu(p)-v p}\rra > 1,
\]
by monotonicity of 
\[
 H\longmapsto \lla \frac{M}{1+H-v p}\rra,
\]
there exists $H(p)>\mu(p)-1$ such that 
\[
\lla \frac{M}{1+H(p)-v p}\rra =1.
\]
But if $p\in \mathrm{Sing}(M)$, such an $H(p)$ cannot be determined and we have 
\[
 H(p)=\mu(p)-1.
\]
Therefore $H(p)$  
is continuous but not $\mathcal{C}^1$ at $p$ such that
%lacks the $\mathcal{C}^1$ character at $p$ such that 
\[
\lla \frac{M}{\mu(p)-v p}\rra =1. 
\]
In addition, the solution of the spectral problem \eqref{specprob} is singular, and a Dirac mass arises at $\pm v_{\max}$. It is even possible to compute the weight of this Dirac mass, see \cite{Caillerie, BouinCaillerie}. Note that when the equilibrium satisfies \eqref{Mv2}, $\text{Sing}(M)=\emptyset$, and that in the case we are considering 
in this part 
\begin{equation}
 \label{Msing}
  M(v)=m\left((1-\dv/2)^2-v^2\right),
 \end{equation}
  $\text{Sing}(M)$ is not empty. Indeed, $M$ vanishes at $\pm (1-\dv/2)$ and Taylor expansions of $M$ at $v=\pm (1-\dv/2)$, show that any $p$ large enough belongs to $\text{Sing}(M)$.

 In this part, we investigate the behavior of the micro-macro scheme 
 %in the case $\text{Sing}(M)\neq\emptyset$, 
 with the equilibrium \eqref{Msing},
 since it was \emph{a priori} not
 designed for such a singular equilibrium function.
 %non-vanishing equilibria satisfiying \eqref{Mv2}, where $\text{Sing}(M)=\nullset$.
 %not designed to take into account such a singular equilibrium function. 
 The consistency of the scheme with the kinetic equation for large values of $\ep\sim 1$ is still clear, since it is written as a finite differences scheme for the kinetic equation. To test the accuracy of the scheme in the asymptotic regime, we first ensure that the reference scheme we use for the limit equation, which has been proposed in \cite{LuoPayneEikonal}, is consistent with the limit equation defined with the constrained Hamiltonian \eqref{DefH}. The right part of Fig.\;\ref{MvHp} presents the numerical  Hamiltonian $H(p)$ computed with the limit scheme. This limit scheme is defined following the idea of \cite{LuoPayneEikonal} with an Euler scheme for \eqref{FinLimEq}, and the constrained Hamiltonian \eqref{DefH} computed with a constrained Newton's method. As expected, it lacks the $\mathcal{C}^1$ character and  the condition $H(p)\ge \mu(p)-1$ is satisfied. The left part of Fig. \ref{MvHp} presents the equilibrium function $M$ 
 \eqref{Msing},
%  that is considered in all the tests of this section. It writes
%  \begin{equation}
%  \label{Msing}
%   M(v)=m(1-v^2),
%  \end{equation}
with $m$ such that $\lla M\rra_{N_v}=1$.
\begin{figure}[!ht]
 \centering
\includegraphics[width=17cm]{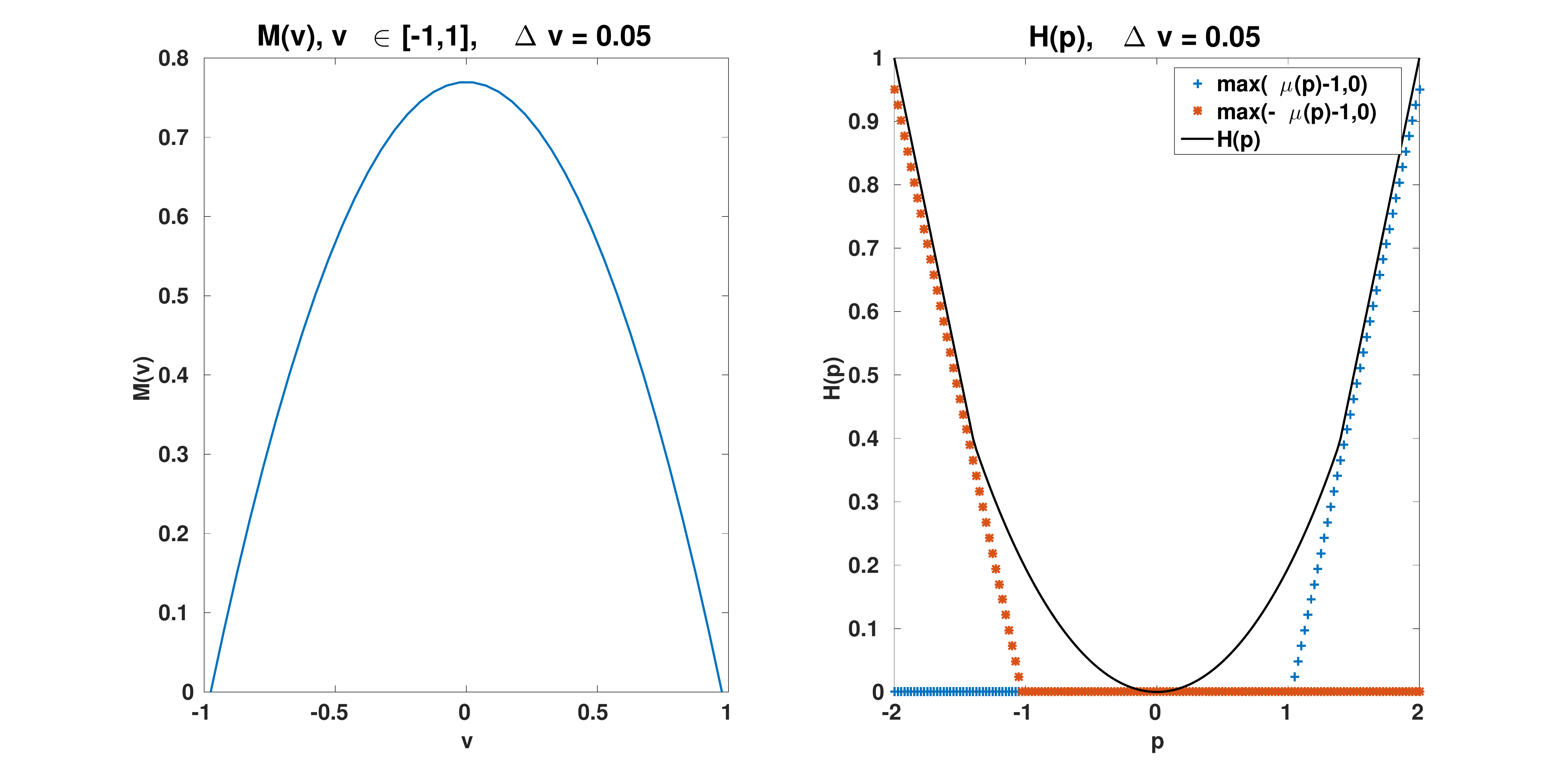}
\caption{Left : The singular equilibrium considered in this section. Right : The constrained Hamiltonian $H(p)$ computed with the limit scheme.}
\label{MvHp}
\end{figure}

The AP character of the micro-macro scheme with this equilibrium is tested in Fig \ref{Phi_ep_10e-4}, where we fixed $\ep=10^{-4}$, $\dx=10^{-2}, \dt=2.5\cdot10^{-3}$, and $\dv=5\cdot10^{-2}$. It presents $\phi$, the Hopf-Cole transform of the spatial density $\rho$,  given by the micro macro scheme, and by the limit scheme. %Even in this case of a singular equilibrium, the curves 
As the two curves match, the scheme seems to enjoy the AP property in the case of a singular equilibrium. It can actually be seen on the expression of the scheme. Indeed, since it can be written
%Remark : linear case : $r=0$ !!!!!
\begin{equation}
\label{machin}
 \left\{
\begin{array}{l}
\displaystyle \frac{\phi^{n+1}_i-\phi^n_i}{\dt}+H^{n+1}_i=0 \vspace{4pt}\\
\displaystyle 
\lla 
%\Bigg\langle
\begin{array}{l}
{}\\
\frac{\displaystyle  M}{
\displaystyle
1
%-\frac{\phi^{n+1}_i-\phi^n_i}{\dt}
+H^{n+1}_i
-\frac{\eta^{n+1}_{i,j}-\eta^n_{i,j}}{\dt}-\left[ v\partial_x (\phi+\eta) \right]^n_{i,j}}
\end{array}
\rra_{N_v} 
%\Bigg\rangle_{N_v}
=1 \vspace{4pt}\\
\displaystyle
1
%-\frac{\phi^{n+1}_i-\phi^n_i}{\dt}
+H^{n+1}_i
-\frac{\eta^{n+1}_{i,j}-\eta^n_{i,j}}{\dt}-\left[v\partial_x(\phi+\eta)\right]^n_{i,j} =\e^{\eta^{n+1}_{i,j}/\ep},
\end{array} 
 \right.
\end{equation}
the third line implies that the denominator of the second line remains positive in the computations. 

Let us now remark that, as it is stated in the continuous case in Thm. \ref{thmCaillerie}, $\eta^{n+1}_{i,j}$ vanishes when $\ep$ goes to $0$. Indeed, if there are at least two index $j\in\ccl 1,N_v\ccr$ such that $M(v_j)\neq 0$, the proof of the discrete maximum principle stated in Section \ref{SecMP} still holds. As a consequence, $\phi^{n+1}_i$ and $\eta^{n+1}_{i,j}$ are uniformly bounded in $\ep$. The third line of \eqref{machin} provides the following inequality 
\[
 1+H^{n+1}_i -\frac{\eta^{n+1}_{i,j}-\eta^n_{i,j}}{\dt}-\left[v\partial_x(\phi+\eta)\right]^n_{i,j} \ge 1+\frac{\eta^{n+1}_{i,j}}{\ep},
\]
that can be recast as follows
\[
 \left(\dt-\ep\right) \eta^{n+1}_{i,j} \ge -\ep\eta^n_{i,j} -\ep\dt\left( [v\partial_x(\phi+\eta)]^n_{i,j} -H^{n+1}_i \right).
\]
As a consequence, $\eta^{n+1}_{i,j}$ is nonnegative when $\ep$ goes to $0$ with the discretization parameters fixed. 
However, if it is positive, the third point of Prop. \ref{propMP} cannot be satisfied in the small $\ep$ limit. 
%But, thanks to the third point of Prop. \ref{propMP}, if it is po
Eventually, 
%Since 
$\eta^{n+1}_{i,j}$ vanishes when $\ep$ goes to $0$.
Thus, the following inequality holds when $\ep\to 0$
\[
 1+H^{n+1}_i-\left[v\partial_x \phi\right]^n_{i,j}\ge 0,
\]
 the constraint $H(\partial_x \phi)\ge \mu(\partial_x \phi)-1$ is hence fulfilled in the limit scheme.

\begin{figure}[!ht]
 \centering
\includegraphics[width=17cm]{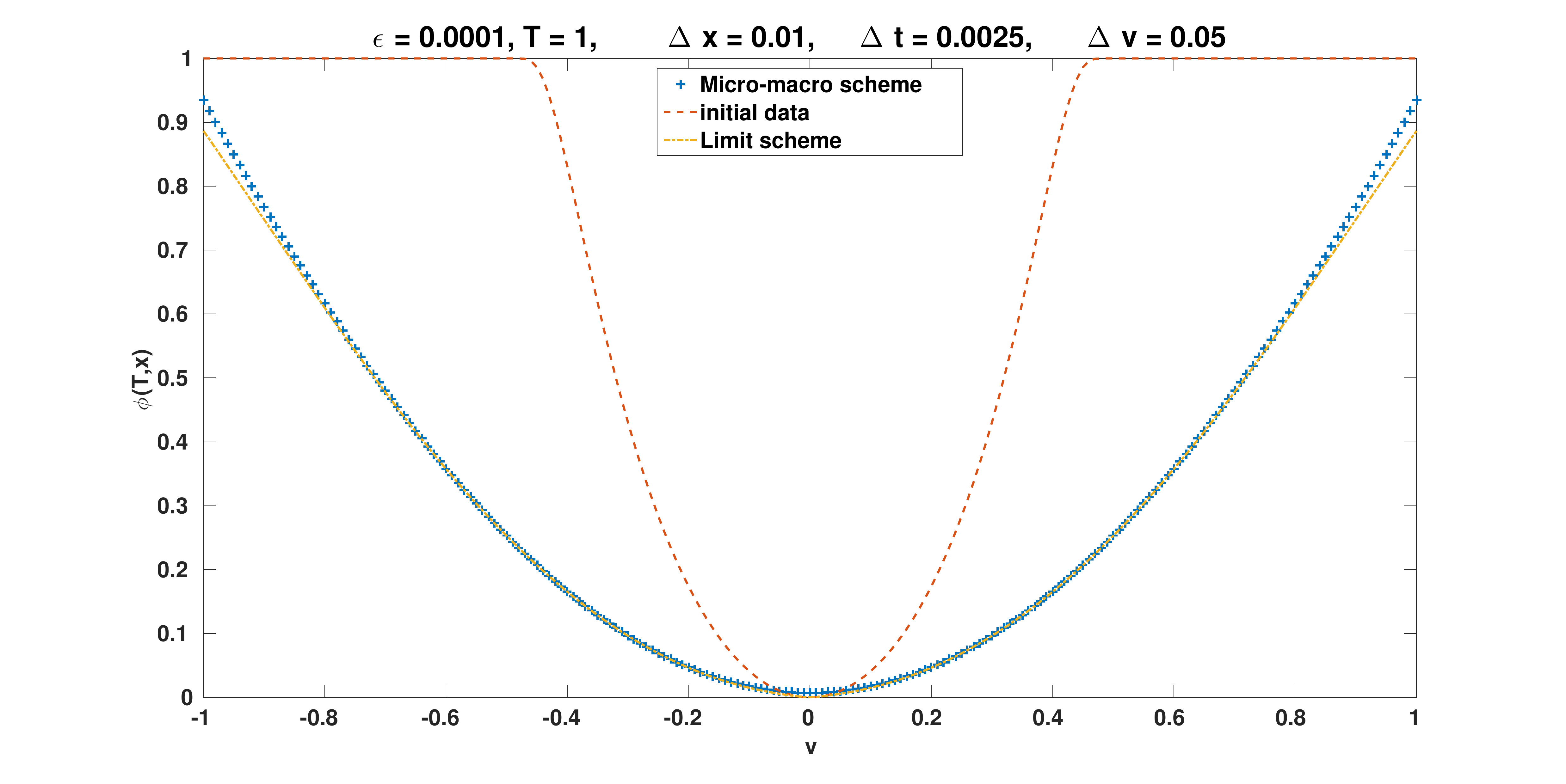}
\caption{The Hopf-Cole transform of the density, $\phi$, given by the micro-macro scheme and the limit scheme at time $T=1$, for $\ep=10^{-4}, \dx=10^{-2}, \dt=2.5\cdot10^{-3}, \dv=5\cdot10^{-2}$.}
\label{Phi_ep_10e-4}
\end{figure}

Moreover, if $v=\pm v_{\max}$, the corrector $\e^{-\eta^\ep(t,x,v)/\ep}$ may not be bounded anymore, and a Dirac mass can even appear for such $v$, see \cite{Caillerie}. This phenomena is highlighted by Fig.\;\ref{ConvergenceDirac}, in which the corrector $\e^{-\eta^\ep(T,x,v)/\ep}$ is represented at time $T=1$ as a function of $x$ and $v$ for different values of $\ep$. As it is expected, we observe that a Dirac mass arises at some points located at the border of the velocity domain. 

\begin{figure}[!ht]
 \centering
\includegraphics[width=17cm]{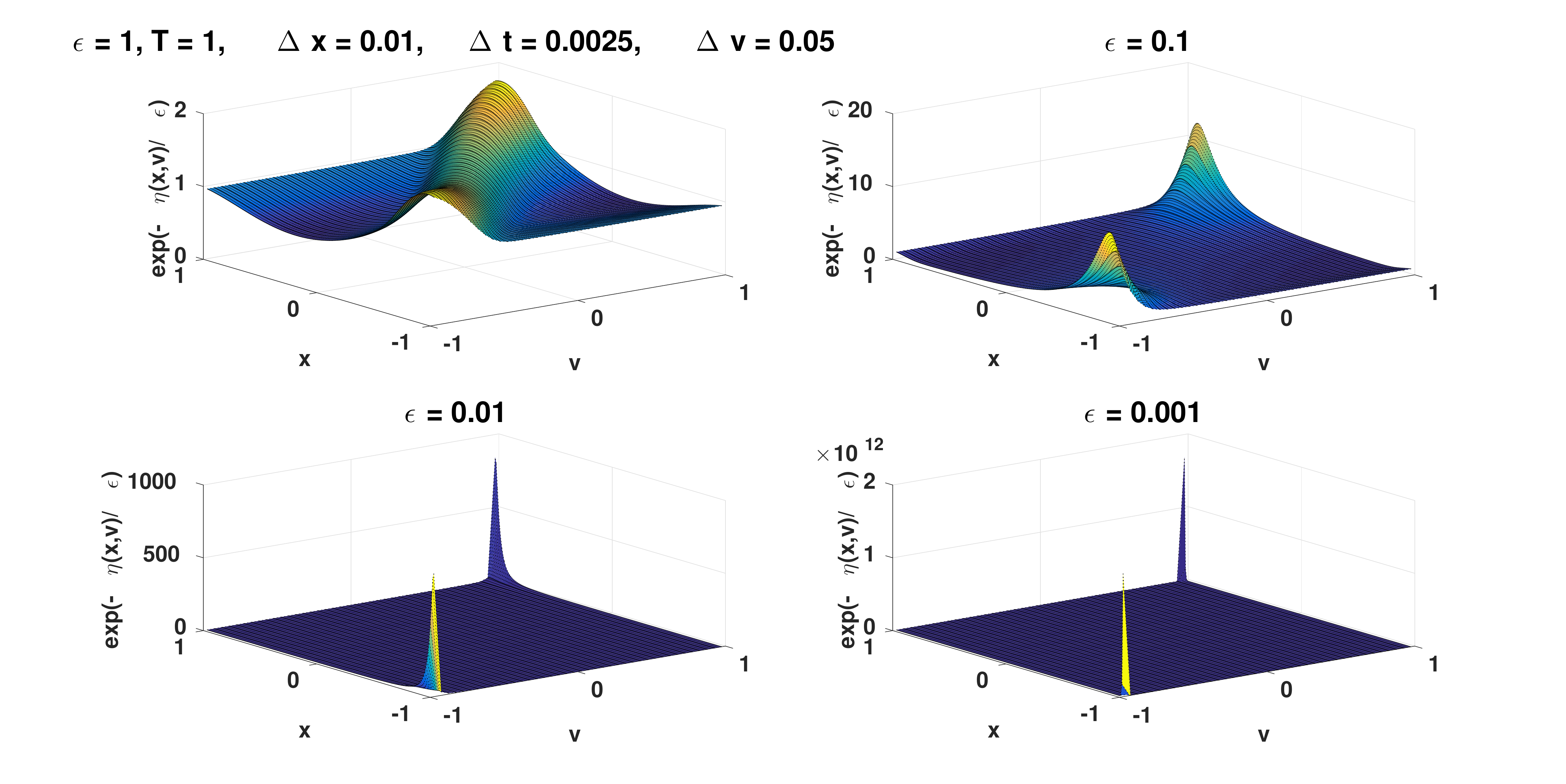}
\caption{The corrector $\e^{-\eta^\ep(T,x,v)/\ep}$ given by the micro-macro scheme at time $T=1$, for different values of $\ep$, and $\dx=10^{-2}, \dt=2.5\cdot10^{-3}, \dv=5\cdot10^{-2}$ as a function of $x$ and $v$. }
\label{ConvergenceDirac}
\end{figure}

As a conclusion, the scheme presented here handles singular measures in the velocity variable which can arise when the equilibrium distribution $M$ vanishes at the boundary of $V$. However, we lack suitable numerical analysis. For instance, the accuracy of the resolution of the non-linear system \eqref{machin} may be reduced when Dirac masses arise in the corrector.

% \vspace{40pt}
% \textbf{Acknowledgements.} This project has received funding from the European Research Council (ERC) under the European Union's Horizon $2020$ research and innovation programme (ERC starting grant MESOPROBIO \no $639638$).

\newpage
%\cite{BouinCalvez}
%%Pour inclure la biblio bib %%
 \bibliographystyle{plain}
\bibliography{PostDoc}
%%\nocite{*}

\end{document}